\documentclass[11pt,a4paper]{article}
\usepackage[T1]{fontenc}
\usepackage[utf8]{inputenc}
\usepackage{lmodern}
\usepackage[margin=27mm]{geometry}
\usepackage{amsmath,amssymb,amsthm,mathtools}
\usepackage{booktabs,array,enumitem}
\usepackage{tikz}
\usetikzlibrary{calc}
\usepackage{microtype}
\usepackage{placeins}
\usepackage[colorlinks=true,linkcolor=blue!45!black,citecolor=blue!45!black,
  urlcolor=blue!45!black]{hyperref}
\usepackage[nameinlink,noabbrev]{cleveref}
\newtheorem{theorem}{Theorem}[section]
\newtheorem{lemma}[theorem]{Lemma}
\newtheorem{corollary}[theorem]{Corollary}
\newtheorem{proposition}[theorem]{Proposition}
\theoremstyle{definition}
\newtheorem{definition}[theorem]{Definition}
\newtheorem{problem}[theorem]{Problem}
\theoremstyle{remark}
\newtheorem{remark}[theorem]{Remark}
\newcommand{\floor}[1]{\left\lfloor #1\right\rfloor}
\setlist[enumerate]{leftmargin=*,itemsep=3pt,topsep=5pt}
\setlist[itemize]{leftmargin=*,itemsep=3pt,topsep=5pt}
\hypersetup{pdftitle={Extremal graphs without cycles of length 0 modulo 4},
 pdfsubject={Constructive classifications for general and 2-connected graphs}}
\title{Extremal graphs without cycles of length 0 modulo 4
\footnote{E-mail addresses:
\texttt{hg@sit.edu.cn} (Y. Hu) and
\texttt{mathdzhang@163.com} (L. Zhang).}}

\author{ Yanan Hu$^{a}$,\quad Leilei Zhang$^{b}$\thanks{Corresponding author.}
\\[3pt]
\small $^{a}$School of Science, Shanghai Institute of Technology, Shanghai 201418, China
\\[2pt]
\small $^{b}$Faculty of Environment and Information Sciences, Yokohama National University,
\\[2pt]
\small Yokohama 240-8501, Japan
}
\date{}

\begin{document}
\maketitle
\vspace{-1.5em}
\begin{abstract}
In a remarkable recent result, Gy\H{o}ri, Li, Salia, Tompkins, Varga and Zhu determined the maximum size in a graph of given order and containing no cycle of length divisible by $4$. In this paper, we give a complete characterization of all graphs attaining this maximum. This answers a question posed by Li, Pan and Shi. Very recently, Chu, Park and Ryu established a sharp upper bound on the number of edges in a 2-connected graph containing no cycle of length divisible by 4. We also give a complete characterization of all such graphs attaining this bound.


\end{abstract}

\section{Introduction}\label{sec:intro}

Unless otherwise stated, all graphs considered in this paper are finite, simple and undirected. We follow the notation and terminology of West~\cite{10}. For a graph $G$, its vertex set and edge set are denoted by $V(G)$ and $E(G)$,
respectively. Its order and size are $v(G)=|V(G)|$ and $e(G)=|E(G)|$. For a family $\mathcal H$ of graphs, a graph is \emph{$\mathcal H$-free} if it contains no member of $\mathcal H$ as a subgraph. The \emph{Tur\'an number} $\operatorname{ex}(n,\mathcal H)$ is the maximum number of edges in an $\mathcal H$-free graph of order $n$. An $\mathcal H$-free graph of order $n$ is \emph{extremal} if it has $\operatorname{ex}(n,\mathcal H)$ edges. Let $K_n$ denote the complete graph of order $n$.  For $r\in\{0,1,2,3\}$, let $\mathcal C_{r\bmod4}$ denote the family of cycles whose lengths are congruent to $r$ modulo $4$.

Determining Tur\'an numbers and characterizing the corresponding extremal graphs are central problems in extremal graph theory~\cite{FS,Turan}. The celebrated Erd\H{o}s--Stone--Simonovits theorem determines the leading term of the Tur\'an number of every fixed graph with chromatic number at least 3. For bipartite graphs, however, the theorem gives only the bound $o(n^2)$. Such problems are called degenerate extremal problems and are often much harder. A prominent example is the Tur\'an problem for $C_4$: although its asymptotic value and exact values for certain orders are known, determining its exact value for every order remains difficult~\cite{Furedi,FS}. This has motivated the study of related problems with other restrictions on cycle lengths.  In particular, the Tur\'an problem for the family of $\mathcal C_{0\bmod4}$ has attracted considerable attention.  For $r\in\{1,2,3\}$, the extremal numbers $\operatorname{ex}(n,\mathcal C_{r\bmod4})$ and all corresponding extremal graphs have been determined~\cite{LPS}. In a remarkable recent result, Gy\H{o}ri, Li, Salia, Tompkins, Varga and Zhu~\cite{GLSTVZ} determined the exact value of $\operatorname{ex}(n,\mathcal C_{0\bmod4})$ for every $n$.

\begin{theorem}[Gy\H{o}ri, Li, Salia, Tompkins, Varga and Zhu~{\cite{GLSTVZ}}] \label{thm:known-general}
Let $n\ge1$ be an integer. Then
\begin{equation}\label{eq:bound}
 \operatorname{ex}(n,\mathcal C_{0\bmod4}) =\floor{\frac{19(n-1)}{12}}.
\end{equation}
\end{theorem}

Gy\H{o}ri et al.~\cite{GLSTVZ} gave constructions attaining this bound at every order, but did not provide a complete characterization of the extremal graphs. Recently, Li, Pan and Shi~\cite{LPS} discussed the extremal numbers $\operatorname{ex}(n,\mathcal C_{r\bmod4})$ and the corresponding extremal graphs for $r\in\{0,1,2,3\}$. For $r=2$, they extended a result of Gao, Li, Ma and Xie~\cite{GLMX}, determining the extremal number at every order and characterizing all extremal graphs. For $r=1$ and $r=3$, the extremal numbers and all extremal graphs follow from the result of F\"uredi and Gunderson~\cite{FG} on $C_5$-free graphs and Mantel's theorem, respectively; see also~\cite{LPS}. For $r=0$, Li, Pan and Shi~\cite{LPS} posed the following problem.

\begin{problem}[Li, Pan and Shi~{\cite{LPS}}]\label{prob:extremal}
Describe all extremal graphs for $\mathcal C_{0\bmod4}$.
\end{problem}

In this paper, we characterize all extremal $\mathcal C_{0\bmod4}$-free graphs, thus resolving Problem~\ref{prob:extremal}. Note that the extremal graphs constructed by Gy\H{o}ri et al.~\cite{GLSTVZ} in Theorem \ref{thm:known-general} have cut vertices for every $n\ge14$. This motivates the study of the corresponding extremal problem for $2$-connected graphs. Recently, Chu, Park and Ryu solved this problem by proving a sharp upper bound.

\begin{theorem}[Chu, Park and Ryu~\cite{CPR}]\label{thm2}
Let $G$ be a $2$-connected graph of order $n$. If $G$ contains no cycle of length divisible by 4, then
\begin{equation}\label{eq:bound-b}
e(G)\le \left\lfloor\frac{3n-1}{2}\right\rfloor.
\end{equation}
\end{theorem}

Chu, Park and Ryu~\cite{CPR} also gave a $2$-connected $n$-vertex $\mathcal C_{0\bmod4}$-free graph attaining \eqref{eq:bound-b} for every $n\ge12$.  For each $n\ge6$, Bai, Chu, Li, Park and Ryu~\cite{BCLPR} determined the maximum size and characterized the unique extremal graph among $2$-connected $n$-vertex $\mathcal C_{2\bmod4}$-free graphs. For $\mathcal C_{1\bmod4}$ and $\mathcal C_{3\bmod4}$, the extremal numbers coincide with those in the unrestricted setting for $n\ge3$ and $n\ge4$, respectively.  However, the problem of characterizing all $2$-connected graphs attaining the bound in Theorem~\ref{thm2} remained open. Thus, our second result characterizes all $2$-connected $\mathcal C_{0\bmod4}$-free graphs attaining the upper bound in (2).

\begin{figure}[htbp]
\centering
\begin{tikzpicture}[
  x=1cm,y=1cm,line width=.65pt,
  every node/.style={font=\small},
  vertex/.style={
    circle,fill=black,inner sep=0pt,minimum size=3.2pt
  },
  every label/.append style={font=\small,inner sep=1pt},
  label distance=2pt
]

\begin{scope}[shift={(1.8,0)}]
  \coordinate (T-x) at (0,2.4);
  \coordinate (T-a) at (-.85,1.8);
  \coordinate (T-b) at (-.85,.6);
  \coordinate (T-y) at (0,0);
  \coordinate (T-c) at (.85,.6);
  \coordinate (T-d) at (.85,1.8);

  \draw (T-x)--(T-a)--(T-b)--(T-y)--(T-c)--(T-d)--cycle
        (T-b)--(T-c);

  \node[vertex,label=above:$x$] at (T-x) {};
  \node[vertex,label=left:$a$] at (T-a) {};
  \node[vertex,label=left:$b$] at (T-b) {};
  \node[vertex,label=below:$y$] at (T-y) {};
  \node[vertex,label=right:$c$] at (T-c) {};
  \node[vertex,label=right:$d$] at (T-d) {};
  \node[font=\normalsize] at (0,-.83) {$T(x,y)$};
\end{scope}

\begin{scope}[shift={(6,0)}]
  \coordinate (S-a) at (-1.15,0);
  \coordinate (S-b) at (1.15,0);
  \coordinate (S-c) at (0,2.3);
  \coordinate (S-z) at (0,.8);
  \coordinate (S-u) at (-.575,.4);
  \coordinate (S-v) at (.575,.4);
  \coordinate (S-w) at (0,1.3);

  \draw (S-a)--(S-b)--(S-c)--cycle
        (S-z)--(S-u)--(S-a)
        (S-z)--(S-v)--(S-b)
        (S-z)--(S-w)--(S-c);

  \node[vertex,label=below left:$a$] at (S-a) {};
  \node[vertex,label=below right:$b$] at (S-b) {};
  \node[vertex,label=above:$c$] at (S-c) {};
  \node[vertex,label=below:$z$] at (S-z) {};
  \node[vertex,label=below:$u$] at (S-u) {};
  \node[vertex,label=below:$v$] at (S-v) {};
  \node[vertex,label=right:$w$] at (S-w) {};
  \node[font=\normalsize] at (0,-.83) {$S_7$};
\end{scope}

\begin{scope}[shift={(10.7,0)}]
  \coordinate (L8-x) at (0,2.4);
  \coordinate (L8-a) at (-1.5,1.95);
  \coordinate (L8-b) at (-1.5,.45);
  \coordinate (L8-y) at (0,0);
  \coordinate (L8-c) at (-.6,.8);
  \coordinate (L8-d) at (-.6,1.7);
  \coordinate (L8-r) at (1.05,.75);
  \coordinate (L8-s) at (1.5,2.4);

  \draw (L8-x)--(L8-a)--(L8-b)--(L8-y)
        --(L8-c)--(L8-d)--cycle
        (L8-b)--(L8-c)
        (L8-x)--(L8-r)--(L8-s)--cycle
        (L8-r)--(L8-y);

  \foreach \v in {a,b,c,d}
    \node[vertex] at (L8-\v) {};
  \node[vertex,label=above:$x$] at (L8-x) {};
  \node[vertex,label=below:$y$] at (L8-y) {};
  \node[vertex,label=right:$r$] at (L8-r) {};
  \node[vertex,label=above:$s$] at (L8-s) {};
  \node[font=\normalsize] at (0,-.83) {$L_8$};
\end{scope}

\begin{scope}[shift={(3.65,-4.85)}]
  \coordinate (L13-x) at (0,2.8);
  \coordinate (L13-y1) at (-2.8,0);
  \coordinate (L13-y2) at (2.8,0);
  \coordinate (L13-p) at (-.7,0);
  \coordinate (L13-q) at (.7,0);

  \coordinate (L13-a1) at (-1.4,2.8);
  \coordinate (L13-b1) at (-2.8,1.4);
  \coordinate (L13-c1) at (-2.35,.95);
  \coordinate (L13-d1) at (-.95,2.35);

  \coordinate (L13-a2) at (1.4,2.8);
  \coordinate (L13-b2) at (2.8,1.4);
  \coordinate (L13-c2) at (2.35,.95);
  \coordinate (L13-d2) at (.95,2.35);

  \draw (L13-x)--(L13-a1)--(L13-b1)--(L13-y1)
        --(L13-c1)--(L13-d1)--cycle
        (L13-b1)--(L13-c1);

  \draw (L13-x)--(L13-a2)--(L13-b2)--(L13-y2)
        --(L13-c2)--(L13-d2)--cycle
        (L13-b2)--(L13-c2);

  \draw (L13-y1)--(L13-p)--(L13-q)--(L13-y2)
        (L13-x)--(L13-p)
        (L13-x)--(L13-q);

  \foreach \v in {a1,b1,c1,d1,a2,b2,c2,d2}
    \node[vertex] at (L13-\v) {};
  \node[vertex,label=above:$x$] at (L13-x) {};
  \node[vertex,label=below:$y_1$] at (L13-y1) {};
  \node[vertex,label=below:$y_2$] at (L13-y2) {};
  \node[vertex,label=below:$p$] at (L13-p) {};
  \node[vertex,label=below:$q$] at (L13-q) {};
  \node[font=\normalsize] at (0,-.83) {$L_{13}$};
\end{scope}

\begin{scope}[shift={(10.5,-4.85)}]
  \coordinate (R-a) at (-1.15,2.3);
  \coordinate (R-b) at (1.15,2.3);
  \coordinate (R-c) at (1.15,0);
  \coordinate (R-d) at (-1.15,0);
  \coordinate (R-h) at (0,1.15);
  \coordinate (R-s) at (1.15,1.15);
  \coordinate (R-t) at (-1.15,1.15);
  \coordinate (R-u) at (-.575,1.725);
  \coordinate (R-v) at (.575,1.725);

  \draw (R-a)--(R-b)--(R-s)--(R-c)
        --(R-d)--(R-t)--cycle
        (R-h)--(R-u)--(R-a)
        (R-h)--(R-v)--(R-b)
        (R-h)--(R-c)
        (R-h)--(R-d);

  \foreach \v in {s,t,u,v}
    \node[vertex] at (R-\v) {};
  \node[vertex,label=above left:$a$] at (R-a) {};
  \node[vertex,label=above right:$b$] at (R-b) {};
  \node[vertex,label=below right:$c$] at (R-c) {};
  \node[vertex,label=below left:$d$] at (R-d) {};
  \node[vertex,label=below:$h$] at (R-h) {};
  \node[font=\normalsize] at (0,-.83) {$R_9$};
\end{scope}

\end{tikzpicture}
\caption{The graphs $T(x,y),S_7,L_8,L_{13}$ and $R_9$.}
\label{fig:special-graphs}
\end{figure}
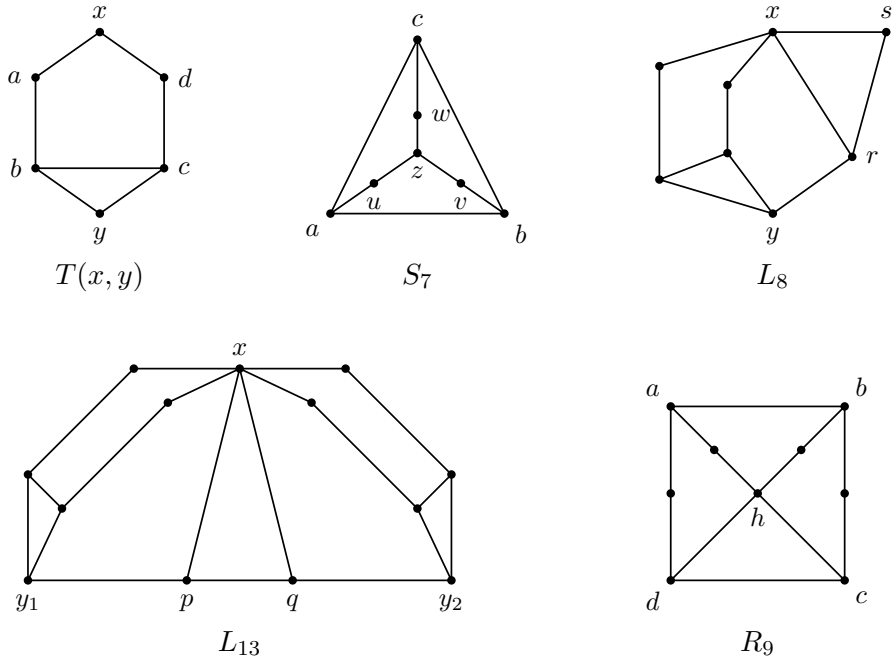

Our first result is as follows. The class $\mathcal G$ is constructed from $K_1,K_2,K_3,S_7,L_8,R_9,L_{13}$ by three operations, specified in Definition~\ref{def:class-G}.  The graphs $T(x,y)$, $S_7$, $L_8$, $L_{13}$ and $R_9$ used in our constructions are shown in Figure~1.

\begin{theorem}\label{thm:main}
Let $G$ be a nonempty $\mathcal C_{0\bmod4}$-free graph. Then $G$ is extremal if and only if $G\in\mathcal G$.
\end{theorem}

To show that our constructions in the $2$-connected case include all extremal graphs, we prove that, up to Whitney switches, every graph attaining the bound is either one of two initial graphs or can be obtained from smaller graphs attaining the same bound. The class $\mathcal B$ is constructed from $L_8$ and $L_{13}$ by the four operations in Definition~\ref{def:class-B}. For convenience, the 2-connected $\mathcal C_{0\bmod4}$-free graphs attaining  equality in \eqref{eq:bound-b} are called \emph{bound-tight}.

\begin{theorem}\label{thm:b}
Let $G$ be a 2-connected $\mathcal C_{0\bmod4}$-free graph. Then $G$ is bound-tight if and only if $G\in\mathcal B$.
\end{theorem}

The remainder of the paper is organized as follows. Section~\ref{sec:preliminaries} collects the preliminary results needed for our proofs. We prove Theorems~\ref{thm:main} and~\ref{thm:b} in Sections~\ref{sec:completeness} and~\ref{sec:two-connected}, respectively.

\section{Preliminaries}\label{sec:preliminaries}

For $S\subseteq V(G)$, let $G[S]$ be the subgraph induced by $S$
and $G-S=G[V(G)\setminus S]$.  For distinct vertices $x,y\in V(G)$ and a subgraph $H$ of $G$, we write $G-x$ for $G-\{x\}$ and $G-H$ for $G-V(H)$.
The graph $G+xy$ is obtained by adding $xy$ if it is absent.
The neighborhood and degree of $x$ are denoted by $N_G(x)$ and $\deg_G(x)$, respectively. The minimum degree of $G$ is denoted by $\delta(G)$.

A \emph{block} is a maximal 2-connected subgraph or a bridge, regarded as $K_2$.
For a connected graph with order $\ge 2$, its \emph{block-cut tree} is the bipartite graph whose vertices are its blocks and cut vertices, with adjacency given by containment. A \emph{terminal block} contains at most one cut vertex. If the graph has more than one block, its terminal blocks are precisely the blocks corresponding to leaves of its block-cut tree; see~\cite{Diestel}.

The length of a path or cycle $P$ is denoted by $\ell(P)$. For $a,b\in V(P)$, the subpath of a path $P$ between $a$ and $b$ is denoted by $P[a,b]$. 
An \emph{$H$-ear} is a path whose only vertices in $H$ are its distinct ends and whose edges lie outside $H$. For distinct vertices $a,b\in V(H)$, a path with ends $a$ and $b$ is called an $(a,b)$-path. The set of residues of the lengths of $(a,b)$-paths in $H$ is denoted by
\[
 \mathcal R(H;a,b)=
 \{\ell(P)\bmod4:P\text{ is an }(a,b)\text{-path in }H\}.
\]
All residues and sums of residue sets are taken modulo 4. A \emph{mixed-parity pair} consists of one even and one odd residue.
For a set $L\subseteq\{0,1,2,3\}$, the pair $(a,b)$ is of \emph{$L$-type} if $\mathcal R(H;a,b)\subseteq L$~\cite{CPR}. The set $L$ may contain residues that do not occur.
Let $\{x,y\}$ be a 2-vertex cut of a 2-connected block $G$.
Let $P$ be a component of $G-\{x,y\}$. The graph $G'$ is obtained from $G$ by first removing all edges between $\{x,y\}$ and $P$, and then adding the edges in
\[
 \{xz:yz\in E(G),\ z\in V(P)\}
 \cup
 \{yz:xz\in E(G),\ z\in V(P)\}.
\]
This operation is called a \emph{Whitney switch}, or simply a \emph{switch}.
This classical switching operation preserves the 2-connectivity of the block \cite[pp.~44--46]{Truemper}.
It also preserves $\mathcal C_{0\bmod4}$-freeness by \cite[Lemma~10]{GLSTVZ}. Two graphs are \emph{switching-equivalent} if one can be obtained from the other by a sequence of Whitney switches. The \emph{switching orbit} of a graph consists of all graphs switching-equivalent to it.

We will use the following results.

\begin{lemma}[Gy\H{o}ri et al.~{\cite{GLSTVZ}}]
\label{lem:planar}
Every $\mathcal C_{0\bmod4}$-free graph is planar.
\end{lemma}

\begin{lemma}[Gy\H{o}ri et al.~{\cite{GLSTVZ}}]
\label{lem:bipartite-bound}
Let $G$ be a bipartite $\mathcal C_{0\bmod4}$-free graph
of order $n\ge4$. Then
\[
e(G)\le\left\lfloor\frac{3(n-2)}{2}\right\rfloor.
\]
\end{lemma}

\begin{lemma}[Bergougnoux et al.~\cite{BBBK}]\label{lem:parity}
For any distinct vertices $a,b$ in a nonbipartite
2-connected graph $B$, the set $\mathcal R(B;a,b)$
contains both an even residue and an odd residue.
\end{lemma}

The two \emph{deficits} corresponding to the general and 2-connected edge bounds are
\[
 D(G)=19(v(G)-1)-12e(G),\qquad
 d(G)=3v(G)-1-2e(G).
\]
A direct calculation gives
\(
 D(G)=v(G)-13+6d(G).
\)
For a nonempty $\mathcal C_{0\bmod4}$-free graph $G$, the general bound of Gy\H{o}ri et al.~\cite{GLSTVZ} gives $D(G)\ge0$.
Since $e(G)$ is an integer,
\[
 0\le D(G)\le11
 \quad\Longleftrightarrow\quad
 e(G)=\left\lfloor\frac{19(v(G)-1)}{12}\right\rfloor.
\]
Thus $G$ is extremal if and only if $D(G)\le11$.
If $G$ is 2-connected, inequality~\eqref{eq:bound-b}
gives $d(G)\ge0$.
If $G$ is bipartite with order $v(G)\ge4$,
Lemma~\ref{lem:bipartite-bound} gives
\[
 d(G)\ge3v(G)-1-3(v(G)-2)=5.
\]

\begin{lemma}\label{lem:sound}
Let $G$ be a $\mathcal C_{0\bmod4}$-free graph. If $G'$ is obtained from $G$ by adding an ear of length $k\ge2$ with distinct ends $a,b\in V(G)$, then $G'$ is $\mathcal C_{0\bmod4}$-free if and only if
$(-k)\bmod4\notin\mathcal R(G;a,b)$.
Moreover,
\[
 D(G')-D(G)=7k-19,\qquad
 d(G')-d(G)=k-3.
\]
\end{lemma}

\begin{proof}
Every cycle of $G'$ that is not contained in $G$
consists of the added ear and an $(a,b)$-path in $G$.
Its length is divisible by 4 if and only if
the path has length congruent to $-k$ modulo 4.
This proves the equivalence.

The ear adds $k-1$ vertices and $k$ edges, so
\[
\begin{aligned}
 D(G')-D(G)&=19(k-1)-12k=7k-19,\\
 d(G')-d(G)&=3(k-1)-2k=k-3.
\end{aligned}
\]
\end{proof}

We use the vertex labels in Figure~\ref{fig:special-graphs}.
The vertices $x,y$ of $T(x,y)$ are its \emph{terminals}, and the other four vertices are its \emph{internal vertices}.
 A copy $T(x',y')$ is attached at distinct vertices $u,v$ by identifying $(x',y')$ with either $(u,v)$ or $(v,u)$. In $S_7$, the vertex $z$ is its \emph{centre}, and the paths $zua,zvb,zwc$ are its \emph{arms}.
The initial graphs used in our constructions are $K_1$ and those listed in Table~\ref{tab:seeds}. All these graphs are $\mathcal C_{0\bmod4}$-free. Among the graphs listed in the table, all except $K_2$
are 2-connected.

\begin{table}[ht]
\centering
\caption{Orders, sizes and deficits of the nontrivial initial graphs.}
\label{tab:seeds}
\begin{tabular}{@{}lrrrr@{}}
\toprule
Graph & Order & Size & $D$ & $d$\\
\midrule
$K_2$&2&1&7&3\\
$K_3$&3&3&2&2\\
$S_7$&7&9&6&2\\
$L_8$&8&11&1&1\\
$R_9$&9&12&8&2\\
$L_{13}$&13&19&0&0\\
\bottomrule
\end{tabular}
\end{table}

The family $\mathcal Z$ consists of the graphs obtained
by the following three constructions.
Each construction starts with two vertex-disjoint copies
$T(x_1,y_1)$ and $T(x_2,y_2)$ of $T(x,y)$.
In each copy, the vertices $x_i,y_i$ correspond to
$x,y$, respectively.
For each $i\in\{1,2\}$, the ordered pair $(u_i,v_i)$
is chosen from
\[
 \{(x_i,y_i),(y_i,x_i)\}.
\]
These choices are independent.
All vertices are distinct except for the stated identification.

\begin{description}[leftmargin=1.4em,labelindent=0pt]
\item[A.]
The vertices $u_1,u_2$ are identified as a single vertex $x$.
Two new vertices $p,q$ are added, together with the edges
$v_1p,pq,qv_2,xp,xq$.

\item[B.]
A new vertex $p$ is added, together with the edges
$u_1u_2,v_1p,pv_2,u_1p,u_1v_2$.

\item[C.]
A new vertex $p$ is added, together with the edges
$u_1u_2,v_1p,pv_2,u_1p,u_2p$.
\end{description}

These constructions are shown in Figure~\ref{fig:templates}. Every graph in $\mathcal Z$ has order 13 and size 19.

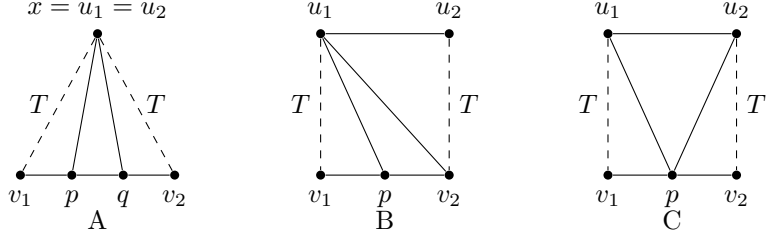
\begin{figure}[htbp]
\centering
\begin{tikzpicture}[
  x=.85cm,y=.85cm,
  vertex/.style={circle,fill=black,inner sep=1.2pt},
  every node/.style={font=\small}
]

\begin{scope}[xshift=0cm]
  \node[vertex,label={above:{$x=u_1=u_2$}}]
    (A-x) at (0,2.2) {};
  \node[vertex,label={below:{$v_1$}}]
    (A-v1) at (-1.2,0) {};
  \node[vertex,label={below:{$p$}}]
    (A-p) at (-.4,0) {};
  \node[vertex,label={below:{$q$}}]
    (A-q) at (.4,0) {};
  \node[vertex,label={below:{$v_2$}}]
    (A-v2) at (1.2,0) {};
  \draw (A-v1)--(A-p)--(A-q)--(A-v2)
        (A-x)--(A-p) (A-x)--(A-q);
  \draw[dashed]
    (A-x)--node[left] {$T$} (A-v1)
    (A-x)--node[right] {$T$} (A-v2);
  \node at (0,-.7) {A};
\end{scope}

\begin{scope}[xshift=3.8cm]
  \node[vertex,label={above:{$u_1$}}]
    (B-u1) at (-1,2.2) {};
  \node[vertex,label={above:{$u_2$}}]
    (B-u2) at (1,2.2) {};
  \node[vertex,label={below:{$v_1$}}]
    (B-v1) at (-1,0) {};
  \node[vertex,label={below:{$p$}}]
    (B-p) at (0,0) {};
  \node[vertex,label={below:{$v_2$}}]
    (B-v2) at (1,0) {};
  \draw (B-u1)--(B-u2)
        (B-v1)--(B-p)--(B-v2)
        (B-u1)--(B-p) (B-u1)--(B-v2);
  \draw[dashed]
    (B-u1)--node[left] {$T$} (B-v1)
    (B-u2)--node[right] {$T$} (B-v2);
  \node at (0,-.7) {B};
\end{scope}

\begin{scope}[xshift=7.6cm]
  \node[vertex,label={above:{$u_1$}}]
    (C-u1) at (-1,2.2) {};
  \node[vertex,label={above:{$u_2$}}]
    (C-u2) at (1,2.2) {};
  \node[vertex,label={below:{$v_1$}}]
    (C-v1) at (-1,0) {};
  \node[vertex,label={below:{$p$}}]
    (C-p) at (0,0) {};
  \node[vertex,label={below:{$v_2$}}]
    (C-v2) at (1,0) {};
  \draw (C-u1)--(C-u2)
        (C-v1)--(C-p)--(C-v2)
        (C-u1)--(C-p) (C-u2)--(C-p);
  \draw[dashed]
    (C-u1)--node[left] {$T$} (C-v1)
    (C-u2)--node[right] {$T$} (C-v2);
  \node at (0,-.7) {C};
\end{scope}

\end{tikzpicture}
\caption{The three constructions defining $\mathcal Z$.}
\label{fig:templates}
\end{figure}

\FloatBarrier
\begin{proposition}\label{prop:switch}
Every member of $\mathcal Z$ is switching-equivalent to $L_{13}$.
In particular, every graph in $\mathcal Z$ is 2-connected
and $\mathcal C_{0\bmod4}$-free.
\end{proposition}

\begin{proof}
The graph $L_{13}$ shown in Figure~\ref{fig:special-graphs}
is obtained by construction A with
$(u_i,v_i)=(x_i,y_i)$ for $i=1,2$.
The two copies of $T$ in each construction are denoted
by $H_1,H_2$.

In construction A, the set $\{x,p\}$ is a 2-vertex cut,
and $V(H_1)\setminus\{x\}$ is the vertex set of a component
after its deletion.
After a switch on this component at $\{x,p\}$,
the first copy has terminals $p,v_1$,
and the edge $pv_1$ becomes $xv_1$.
The second copy still has terminals $x,v_2$.
The edges outside the two copies are now
\[
 xv_1,\ qv_2,\ pq,\ xq,\ xp.
\]
With the second copy listed first, the relabelling
\(
 (u_1,v_1,u_2,v_2,p)_{\mathrm B}
   =(x,v_2,v_1,p,q)_{\mathrm A}
\)
gives a graph of construction B.

In this graph of construction B, the set $\{u_1,p\}$
is a 2-vertex cut, and $V(H_1)\setminus\{u_1\}$
is the vertex set of a component after its deletion.
After a switch on this component at $\{u_1,p\}$,
the first copy has terminals $p,v_1$,
and the edge $pv_1$ becomes $u_1v_1$.
The second copy still has terminals $u_2,v_2$.
The edges outside the two copies are now
\[
 u_1v_1,\ u_1u_2,\ pv_2,\ u_1p,\ u_1v_2.
\]
The relabelling
\(
 (u_1,v_1,u_2,v_2,p)_{\mathrm C}
   =(p,v_1,v_2,u_2,u_1)_{\mathrm B}
\)
gives a graph of construction C.

In each construction, deleting the two terminals of either copy separates its four internal vertices from the rest of the graph. These vertices induce a path and therefore form a component. A switch on this component at $\{u_i,v_i\}$ exchanges the neighbors of $u_i,v_i$ within the copy, leaving all edges outside the copy unchanged. The correspondence $x\mapsto u_i$, $y\mapsto v_i$ is replaced by $x\mapsto v_i$, $y\mapsto u_i$, or conversely. Thus the switch realizes the other choice in the definition of $\mathcal Z$, while $u_i,v_i$ retain their roles in the construction. These switches can be performed independently for the two copies.
Consequently, every graph specified by  constructions A, B and C is switching-equivalent to $L_{13}$.

As noted above, $L_{13}$ is 2-connected and
$\mathcal C_{0\bmod4}$-free. Since Whitney switches preserve both properties, every graph in $\mathcal Z$ has these properties as well.
\end{proof}

\begin{remark}\label{rem:l8-switch}
In the construction of $L_8$, either of the two identifications
of the terminals of $T(x',y')$ with $x,y$ yields a graph
switching-equivalent to $L_8$.
Indeed, the four internal vertices of the copy form
a component after $x,y$ are deleted.
A switch on this component at $\{x,y\}$ transforms
one graph into the other.
Since $L_8$ is 2-connected and $\mathcal C_{0\bmod4}$-free,
both graphs have these properties.
\end{remark}

A \emph{theta} is the union of three internally vertex-disjoint paths with the same two ends. It is an \emph{even theta} if 
all three paths have even length. Following~\cite{GLSTVZ}, we know
a $\mathcal C_{0\bmod4}$-free graph contains no even theta.
A \emph{thread} is a maximal path whose internal vertices have
degree 2 and whose ends have degree different from 2. A thread $P$ is \emph{suppressed} by deleting all its edges and internal vertices and adding a single edge joining its ends.
Any parallel edges arising from this operation are retained.

A \emph{triangle-star} in $G$ is a copy of $S_7$ whose centre has degree 3 in $G$ and whose three internal arm vertices have degree 2 in $G$. For a triangle-star with centre $z$ and arms $zua,zvb,zwc$, the operation replacing $G$ by $G-\{z,u,v,w\}$
is called a \emph{triangle-star deletion}. A graph $G$ is \emph{locally reduced throughout its switching orbit}
if no graph in its switching orbit contains a triangle-star or a pair of adjacent vertices of degree 2.

\begin{lemma}
\label{lem:triangle-star}
If $H$ is obtained from a 2-connected
$\mathcal C_{0\bmod4}$-free graph $G$ by a triangle-star
deletion, then $H$ is 2-connected and
$\mathcal C_{0\bmod4}$-free, with
\[
 d(H)=d(G),\qquad D(H)=D(G)-4.
\]
Moreover, either $H=K_3$ or $d(G)\ge4$.
\end{lemma}

\begin{proof}
The deleted triangle-star has triangle $C=abca$
and arms $zua,zvb,zwc$, so $H=G-\{z,u,v,w\}$.
The deleted vertices have neighbors in $H$ only at $a,b,c$.
Every component of $H$ therefore meets $C$;
otherwise, it would also be a component of $G$.
Since $C$ is connected, $H$ is connected.

For any vertex $t\in V(H)$, the surviving vertices of $C$
lie in one component of $H-t$.
Any other component would have no neighbor among
$z,u,v,w$ and would therefore also be a component of $G-t$,
contrary to the 2-connectivity of $G$.
Thus $H$ is 2-connected.
As a subgraph of $G$, it is $\mathcal C_{0\bmod4}$-free.
The deletion removes 4 vertices and 6 edges,
which gives the two deficit identities.

A path in $H$ between distinct vertices of $C$
cannot have length $0$ modulo 4:
the two corresponding arms through $z$ would complete
a cycle of length divisible by 4.
For a $C$-ear $P$ in $H$, the edge and the two-edge arc
of $C$ joining its ends also rule out lengths
$3$ and $2$ modulo 4, respectively.
Hence every $C$-ear in $H$ satisfies
\[
 \ell(P)\equiv1\pmod4.
\]
Since its edges lie outside the triangle $C$,
such an ear cannot have length 1.
Consequently, $\ell(P)\ge5$.

For any component $S$ of $H-V(C)$, at least two vertices
of $C$ have neighbors in $S$, since $H$ is 2-connected.
If all three vertices have neighbors in $S$,
a tree in $S$ joining one chosen neighbor of each of
$a,b,c$, together with the three corresponding edges to $C$,
contains a minimal subtree $Q$ spanning $a,b,c$.
Its only leaves are $a,b,c$, and all its other vertices
lie in $S$.
The three paths in $Q$ between its leaves are $C$-ears
and therefore have odd length.
Their lengths have an odd sum, whereas every edge of $Q$
is counted twice in this sum.
This is a contradiction.
Thus exactly two vertices $x,y$ of $C$ have neighbors in $S$.

The graph \(  N_S=H[V(S)\cup\{x,y\}] \)
contains the edge $xy$ and is 2-connected.
Indeed, for every $t\in V(N_S)$, the surviving vertices of $\{x,y\}$ lie in one component of $N_S-t$. When both survive, the edge $xy$ joins them. Any other component contains neither surviving vertex
of $\{x,y\}$ and would also be a component of $H-t$, contradicting the 2-connectivity of $H$.

Every $(x,y)$-path in $N_S$ is either the edge $xy$
or a $C$-ear in $H$, and hence has odd length.
Lemma~\ref{lem:parity} implies that $N_S$ is bipartite.
Since $S$ is connected and has neighbors at both $x,y$,
there is a $(x,y)$-path with all internal vertices in $S$.
This path is a $C$-ear of length at least 5,
so $v(N_S)\ge6$.
Lemma~\ref{lem:bipartite-bound} therefore gives
\[
 e(N_S)\le\frac{3(v(N_S)-2)}2
          =\frac32|V(S)|.
\]

If $H-V(C)$ has $k$ components, counting the three
edges of $C$ only once gives
\[
 e(H)
 =3+\sum_S\bigl(e(N_S)-1\bigr)
 \le3+\frac32\bigl(v(H)-3\bigr)-k.
\]
It follows that
\[
 d(G)=d(H)\ge2+2k.
\]
For $k=0$, the graph $H$ is $K_3$;
for $k\ge1$, the inequality gives $d(G)\ge4$.
\end{proof}

A \emph{subdivision of $K_4$} is obtained by replacing each edge of $K_4$ with a path of positive length. The internal vertices of these paths are new, and the paths have pairwise disjoint sets of internal vertices.
The four original vertices are its \emph{branch vertices}. For a multigraph $M$, a \emph{series reduction} replaces two distinct edges $xu,uy$, where $\deg_M(u)=2$, by a single edge $xy$ and deletes $u$; the case $x=y$
produces a loop. A \emph{parallel reduction} removes one of two parallel edges with distinct ends. Each reduction involves both participating edges. Cycles of length 2 formed by parallel edges are allowed
in multigraphs.

The following lemma is a consequence of
\cite[Theorem~1 and Corollary~4]{Duffin}.

\begin{lemma}\label{lem:duffin}
If a connected loopless multigraph $M$ with $e(M)\ge2$
contains no subdivision of $K_4$ and every two distinct
edges of $M$ lie on a common cycle, then $M$ can be reduced
to a single loop by series and parallel reductions.
For every $e_0\in E(M)$, the sequence can be chosen so that
$e_0$ is involved only in the final reduction.
\end{lemma}

\begin{lemma}\label{lem:bipartite}
If $G$ is a 2-connected bipartite $\mathcal C_{0\bmod4}$-free
graph, then $G$ contains three pairwise vertex-disjoint edges
whose endpoints all have degree 2 in $G$.
\end{lemma}

\begin{proof}
The graph $G$ contains no subdivision of $K_4$.
Indeed, two of its four branch vertices would lie in the
same bipartition class, and the three internally disjoint
paths between them would form an even theta.

If $G$ is a cycle, then $v(G)\ge6$, and three alternating
edges give the conclusion.
Otherwise, every thread of $G$ has distinct ends of
degree at least 3.
The suppression of all threads produces a loopless
connected multigraph $M$ with no cut vertex and
$\delta(M)\ge3$.
Each edge of $M$ corresponds to a thread of $G$,
and distinct threads have disjoint interiors.
Any subdivision of $K_4$ in $M$ would expand to one in $G$,
so $M$ contains no subdivision of $K_4$.
Moreover, every two distinct edges of $M$ lie on a common
cycle, since this property holds in $G$ and is preserved
by thread suppression.

By Lemma~\ref{lem:duffin}, the graph $M$ admits a
reduction sequence, in which each step decreases the
number of edges by 1.
The number of edges of $M$ having a parallel edge
is denoted by $q$.
Since $\delta(M)\ge3$, the first reduction must involve
two parallel edges, so $q\ge2$.

If $q=2$, then $M$ has a unique pair of parallel edges
$e_0,e_1$.
By Lemma~\ref{lem:duffin}, a reduction sequence can be
chosen so that $e_0$ is involved only in the final reduction.
However, the first reduction must involve $e_0,e_1$.
Since $e(M)\ge3$ and each reduction decreases the number
of edges by 1, the first reduction cannot be the last.
This contradiction gives $q\ge3$.

For two vertices $x,y$ joined by exactly $r\ge2$ edges
of $M$, the corresponding threads are denoted by
$P_1,\ldots,P_r$, with internal vertex sets
$I_1,\ldots,I_r$.
Any two of these threads can be completed to a theta in $G$.
If $v(M)=2$, then $r\ge3$ by $\delta(M)\ge3$,
and a third thread supplies the remaining path.
If $v(M)\ge3$, the underlying simple graph $M_0$ of $M$
is 2-connected, so $M_0-xy$ contains an $(x,y)$-path.
Its expansion in $G$ is internally disjoint from
$P_1,\ldots,P_r$ and supplies the third path.

The three paths of each such theta have the same parity
because $G$ is bipartite.
They must be odd, since an even theta is forbidden.
Furthermore,
\[
 \ell(P_i)+\ell(P_j)\not\equiv0\pmod4
 \qquad(1\le i<j\le r),
\]
since $P_i\cup P_j$ is a cycle.
Thus there is a residue $\varepsilon\in\{1,3\}$ such that
\[
 \ell(P_i)\equiv\varepsilon\pmod4
 \qquad(1\le i\le r).
\]

Each $G[I_i]$ contains a matching of size
$\lfloor(\ell(P_i)-1)/2\rfloor$.
Since the sets $I_i$ are pairwise disjoint,
these matchings together have size
\[
 m=\sum_{i=1}^{r}
 \left\lfloor\frac{\ell(P_i)-1}{2}\right\rfloor,
\]
and all their endpoints have degree 2 in $G$.

If $\ell(P_i)\ge3$ for every $i$, then $m\ge r$.
Otherwise, exactly one thread has length 1,
since $G$ is simple.
The common residue is then $\varepsilon=1$,
so every other thread has length at least 5.
Consequently,
\[
 m\ge2(r-1)\ge r.
\]

The threads associated with distinct unordered pairs
of ends have disjoint interiors.
Each of the $q$ edges counted above belongs to exactly
one such pair.
The resulting matchings therefore have disjoint vertex
sets and together contain at least $q\ge3$ edges,
whose endpoints all have degree 2 in $G$.
\end{proof}

For a 2-vertex cut $S=\{x,y\}$ of $G$, a \emph{side} is a subgraph $N=G[V(H)\cup S]$, where $H$ is a component of $G-S$. A cut is \emph{good} if every side contains an odd cycle.

\begin{lemma}\label{lem:cuts}
If $G$ is a 2-connected $\mathcal C_{0\bmod4}$-free graph
with no adjacent vertices of degree 2 and $S=\{x,y\}$
is a 2-vertex cut of $G$, then the following hold.
\begin{enumerate}[label=\rm(\roman*)]
\item Every side $N$ at $S$ with $v(N)\ge4$
contains an odd cycle.
\item The graph $G-S$ has exactly 2 components $H_1,H_2$.
\end{enumerate}
In particular, $S$ is good if $xy\in E(G)$ or
$\min\{v(H_1),v(H_2)\}\ge2$.
\end{lemma}
\begin{proof}
For a side $N$ at $\{x,y\}$, the blocks of $N$ form
a chain $B_1,\ldots,B_k$ from $x$ to $y$,
since $N+xy$ is 2-connected.
The connecting vertices are denoted by
$z_0=x,z_1,\ldots,z_k=y$, where
\[
 V(B_i)\cap V(B_{i+1})=\{z_i\}
 \qquad(1\le i<k).
\]
In particular,
\[
 \deg_{B_i}(v)=\deg_G(v)
 \qquad
 \bigl(v\in V(B_i)\setminus\{z_{i-1},z_i\}\bigr).
\]

(i) If $N$ is bipartite and $v(N)\ge4$, then every
block $B_i$ is bipartite.
If some $B_i$ is 2-connected, Lemma~\ref{lem:bipartite}
gives three pairwise vertex-disjoint edges whose endpoints
have degree 2 in $B_i$.
At least one of these edges avoids $z_{i-1},z_i$,
so both its endpoints have degree 2 in $G$,
a contradiction.
Otherwise, all blocks of $N$ are bridges.
Thus $N$ is an $(x,y)$-path with at least two internal
vertices, again giving adjacent vertices of degree 2 in $G$.
This proves (i).

(ii) If $v(N)\ge4$, then (i) gives a nonbipartite
2-connected block $B_i$.
By Lemma~\ref{lem:parity}, there are two
$(z_{i-1},z_i)$-paths in $B_i$ of opposite parities.
Fixed paths through the remaining blocks extend them
to two $(x,y)$-paths in $N$ of opposite parities.
Hence $N$ contains an even $(x,y)$-path.
If $v(N)=3$, its unique vertex outside $\{x,y\}$
supplies an $(x,y)$-path of length 2.

Thus every side contains an even $(x,y)$-path.
Paths chosen from distinct sides are internally
vertex-disjoint, so three components of $G-\{x,y\}$
would give an even theta.
Since $\{x,y\}$ is a cut, $G-\{x,y\}$ therefore has
exactly two components.

If both components have at least two vertices,
then (i) shows that the cut is good.
If $xy\in E(G)$, every side with only one vertex
outside the cut is a triangle.
Together with (i), this proves the final assertion.
\end{proof}

The next lemma identifies the side beyond a maximal good cut. We give a proof following the approach of
\cite[Claim~5 in the proof of Theorem~1]{GLSTVZ}.

\begin{lemma}\label{lem:terminal-side}
If $G$ is a 2-connected $\mathcal C_{0\bmod4}$-free graph
with no adjacent vertices of degree 2 throughout its
switching orbit, $S_0$ is a good cut of $G$,
$B_0$ is a component of $G-S_0$, and $S$ is a good cut
disjoint from $V(B_0)$ chosen to maximize the order
of the component $B$ of $G-S$ containing $B_0$,
then $G-B$ is either a triangle with terminal set $S$
or a copy $T(x',y')$ with $\{x',y'\}=S$.
\end{lemma}

\begin{proof}
For $S=\{x,y\}$, the two sides $H=G-B$ and
$J=G[V(B)\cup S]$ both contain an odd cycle.

\smallskip
\noindent\textit{Claim 1. The only good cut contained in $V(H)$ is $S$.}

Any other such cut avoids $B_0$ and leaves $B$ intact.
Its component containing $B_0$ also contains a vertex of $S$,
contrary to maximality.

The blocks of $H$ form an $x$--$y$ chain. The two chain ends
of a nonbipartite block give a good cut by
Lemma~\ref{lem:cuts}, so Claim~1 forces this block to be $H$.
Thus $H$ is $2$-connected. If $H$ has no even cycle, it is an
odd cycle, since any additional ear creates a theta and hence
an even cycle. Only $x,y$ have neighbors outside $H$, so
absence of adjacent $2$-degree vertices gives $H=K_3$.

Following~\cite{GLSTVZ}, an \emph{adjustable $(u,v)$-path}
consists of an odd cycle and two vertex-disjoint paths from $u,v$ to distinct vertices of the cycle, each meeting
the cycle only at its endpoint; either path may be trivial.
Each side at a good cut supplies
an adjustable path between its roots, by its block chain
and Menger's theorem. An adjustable $C$-ear meets $C$ only
at its ends and uses no edge of $C$; ears may be chords.
The notation $C[a,b]$ refers to a fixed orientation of $C$,
and span parity means the parity of either arc between the ends.

Among even cycles $C\subseteq H$, the chosen one maximizes
the component $W$ of $G-C$ containing $B$, then minimizes
$\ell(C)$. Hereafter $A=N_G(W)\cap V(C)$.
No even cycle disjoint from $W$ omits a vertex of $A$,
and no even cycle has its vertex set properly contained
in $V(C)$. Hence $C$ has at most one chord, of even span.
An odd-span chord gives a smaller even cycle; two chords
give a theta omitting a cycle vertex. For alternating ends,
one omits a nonempty open arc; otherwise one chord shortcuts
the cycle while retaining the other.
Two disjoint $S$--$C$ paths in $H$, combined with a suitable
root path in $J$, give an even $C$-ear through $W$.
If $xy\in E(C)$, the even root path avoids that edge.

\smallskip
\noindent\textit{Claim 2. $G-C$ is connected.}

For a component $U\ne W$, its attachment set is
$T=N_G(U)\cap V(C)$. If $|T|\ge3$, a tripod through $U$
at any three vertices of $T$ forces $A$ into that triple:
otherwise its other two cycle arcs give a theta omitting
an attachment. Thus $|T|=3$. If $A=T$, the three sector
cycles must all be odd, although their lengths sum to
$\ell(C)$ plus twice the total branch length. Hence $|A|=2$.
Its sides contain $J$ and at least four vertices, respectively,
so Lemma~\ref{lem:cuts} makes $A$ good.
Claim~1 gives $A=S$ and $W=B$, with $T=\{x,y,z\}$.

If $xy\notin E(C)$, connectedness of $H-S$ gives an ear $P$
between the open $(x,y)$-arcs, avoiding $U$, which attaches
there only at $z$. On $C-y$, its end interval overlaps that
of the tripod's $(x,z)$-path; the nested or alternating
order of their ends gives a theta avoiding $y$.
Thus $xy\in E(C)$. The other $(x,z,y)$-arc has length at
least $5$, with, say, its $(x,z)$-segment of length at least
$3$. Absence of adjacent $2$-degree vertices supplies an
internal vertex incident with an edge outside $C$.
This edge extends to a further $C$-ear $P$ outside $W\cup U$.
The path $C-x$ and the tripod's $(z,y)$-path apply when its
other end is $y$; otherwise $C-y$ and the $(x,z)$-path apply.
Overlapping end intervals again give a theta avoiding an attachment.

Consequently every $U\ne W$ has two attachments $s,t$.
The set $\{s,t\}$ differs from $S$, since equality would give at least
three components after deleting $S$.
Unless $U=\{u\}$ and $st\notin E(G)$,
Lemma~\ref{lem:cuts} makes $\{s,t\}$ good, contrary to
Claim~1. Thus $E=sut$ is even. Its ends have odd distance
on $C$, since otherwise $C\cup E$ is an even theta.
Lemma~\ref{lem:cuts} supplies a crossing ear $Q$, with ends
in the order $s,r,t,q$. An interior in another component
outside $W$ would make $Q$ even. Both ears then have odd
span, so the cycles using both ears and opposite arc pairs
are even. Their lengths sum to $2\pmod4$, a contradiction.
Thus two vertex-disjoint even $C$-ears cannot alternate.

If $Q=rq$ is a chord, a theta omitting any point of the four
open arcs gives $A\subseteq\{s,r,t,q\}$.
Their arc parities are $(0,1,1,0)$ or $(1,0,0,1)$.
For an even ear $R$ through $W$, ends $r,q$ give alternating
even ears,
and ends $s,t$ give an even theta with $E$ and a route
through $Q$. For consecutive ends, say $s,r$, an even
intervening arc gives an even theta with $C$. Otherwise the routes
through $q,Q$ and through $E,t$ are even and form an even
theta with $R$.

Hence $Q-C\subseteq W$.
If $S\subseteq V(C)$, then $W=B$, $A=S$, and an even
root path in $J$ crosses $E$, again giving alternating even ears.
Otherwise an odd cycle $O\subseteq J$ meets $C$ at most
once. No three vertex-disjoint $O$--$C$ paths exist, since
each meets $S$.
There is an adjustable $C$-ear $R$ containing $O$ with end
$r$. If $O\cap C=\{c\}$, a path through $W$ joins $r$ to
$O-c$, retaining $c$; when $r=c$, another attachment replaces
$r$ in this construction. If $O$ avoids $C$, all edges within
$C$ are deleted from $G[V(W)\cup V(C)]$, and $C-r$ is
identified to $t_0$. No vertex separates $O$ from
$\{r,t_0\}$: a separator in $W$ would be a cut vertex of
$G$; deleting $r$ leaves a route to $C-r$, and deleting
$t_0$ leaves a route through $W$ to $r$.
Menger's theorem supplies the two arms of $R$.

If $R$ ends at $q$, its even path crosses $E$.
Otherwise a path through $W$ joins $q$ to $R-C$, first at
$z$. If $z$ lies on the $r$-arm, including its end on $O$,
their union contains an $(r,q)$-ear $K$ crossing $E$ and an
adjustable $(z,h)$-path from the remainder of $R$, otherwise
disjoint from $C\cup E\cup K$.
The ear $K$ is odd. Unless $h\in\{s,t\}$, the adjustable
path and a suitable segment of $K$ give an even ear crossing
$E$. For $h=s$ (the other case is symmetric), the adjustable
ears from $s$ to $r,q$ force $C[s,r],C[q,s]$ odd and the
other two arcs even. The odd segment of $K$ and its adjacent
odd arc give an even $(s,z)$-path; the remaining segment,
the other even arc, and $E$ give a second. The adjustable
$(z,s)$-path supplies a third, an even theta.
If $z$ lies off the $r$-arm, their union gives an adjustable
$(r,q)$-path containing $O$. Any third intersection with
$C$ gives three disjoint $O$--$C$ paths. Its even path
therefore crosses $E$, proving Claim~2.

\smallskip
The alternating classes $X,Y$ of $C$ each contain at most
two attachments: a tripod at three same-class attachments
has two branches of the same parity, giving an even ear
of even span. Both classes occur, by the even ear through
$W$. If exactly one root $x$ lies on $C$, all other
attachments are opposite to $x$: a path from $y$ through
$H-C$ to such an attachment combines with $J$ into an
adjustable ear. If both roots lie on $C$, Claim~2 gives $A=S$.

For a same-class pair $A_0\subseteq A$, no root belongs
to $A_0$. The graph $F$ consists of $H-C$, $A_0$, and only
the edges from $A_0$ to $H-C$.
Each attachment reaches $S'=S\setminus V(C)$ without the
other. Two disjoint $(S',A_0)$-paths would combine with $J$
into an even ear of even span (and are impossible if $|S'|=1$).
Menger gives a separator $a\notin V(C)$ maximizing the set
$R_a$ reachable from $S'\setminus\{a\}$ in $F-a$.
There is an $a$--$A_0$ $2$-fan in $F-R_a$; otherwise
another separator $a'\ne a$ gives
$R_{a'}\supseteq R_a\cup\{a\}$.

\smallskip
\noindent\textit{Claim 3. $|A|\le3$.}

If both classes contain two attachments, their separators
$a,b$ are distinct; equality would give a cut vertex.
The cut $\{a,b\}$ separates $B$ from $C$ and is good by
Lemma~\ref{lem:cuts}. Claim~1 gives $\{a,b\}=S$, say
$a=x,b=y$. Their fans are disjoint, as an intersection
bypasses a separator. They give odd $C$-ears $P_X,P_Y$
through $x,y$. Alternating ends, together with $C$ and a
root path in $J$, yield a $K_{3,3}$-subdivision, contradicting
planarity. Otherwise their ends occur as $a_1,a_2,b_1,b_2$.
The odd cycles $P_X\cup C[a_1,a_2]$, $P_Y\cup C[b_1,b_2]$,
the arc $C[a_2,b_1]$, and $J$ give three adjustable paths
meeting cyclically only at $x,b_1,y$. Each supplies two
consecutive length residues modulo $4$; their threefold
sum covers all residues, giving a forbidden cycle.

At most five vertices of $C$ are attachments or chord ends;
all others have degree $2$. Thus $\ell(C)\le10$.
Equality forces the five $2$-degree vertices to alternate
with the rest, putting $A$ in one class. Hence $\ell(C)=6$.
It has a chord: otherwise absence of adjacent $2$-degree
vertices forces three alternating attachments, again in one class.

\smallskip
\noindent\textit{Claim 4. $|A|=2$.}

If $|A|=3$, the cycle may be labeled
$C=x_1y_1x_2y_2x_3y_3x_1$ with $A\cap Y=\{y_3\}$.
The fan separator $a$ for $A\cap X$ and $y_3$ form a good
cut. Claim~1 gives $S=\{a,y_3\}$, with $a=x$.
If $A\cap X=\{x_1,x_3\}$, preventing $\{x_1,x_3\}$
from being a new good cut forces the chord to be $y_1y_3$
or $y_2y_3$, leaving adjacent $2$-degree vertices.
Thus, after reflection, $A\cap X=\{x_1,x_2\}$.
Similarly, the cut $\{x_2,y_3\}$ forces the chord to be
$x_1x_3$ or $y_1y_2$.

The fan paths $P_i$ from $x$ to $x_i$ have lengths $p_i$
with $p_1+p_2$ odd, since otherwise they form an even theta
with $C$. A chord $x_1x_3$ gives three
adjustable paths cyclically at $x,x_1,y_3$, using the
triangle $x_1y_3x_3x_1$, the odd cycle
$P_1\cup P_2\cup x_1y_1x_2$, and $J$.
For the chord $y_1y_2$, odd $p_1$ gives an even theta
consisting of $P_1+x_1y_3$, $P_2+x_2y_1y_2x_3y_3$,
and an even $(x,y_3)$-path in $J$.
For even $p_1$, its three paths are $P_1$,
$P_2+x_2y_2y_1x_1$, and an odd root path in $J$ followed
by $y_3x_1$. Both cases are impossible, proving Claim~4.

Finally, $A$ is good, so Claim~1 gives $A=S$ and $W=B$.
Thus $H$ is the six-cycle with its unique chord, say $y_1y_3$.
Its $X$-terminal is $x_1$, since otherwise $\{y_1,y_3\}$
cuts off the triangle $y_1x_1y_3y_1$ and is a new good cut.
Its $Y$-terminal is $y_2$, since either other choice leaves
adjacent $2$-degree vertices. Hence $H$ is the required copy of $T$.
\end{proof}

\begin{lemma}\label{lem:structure}
If a 2-connected $\mathcal C_{0\bmod4}$-free graph $G$
has a good cut and no graph switching-equivalent to $G$
has adjacent vertices of degree 2, then $G$ contains
induced subgraphs $Q_1,Q_2$ and vertex-disjoint paths
$P_1,P_2$ with the following properties.
\begin{enumerate}[label=\rm(\roman*)]
\item
For $i\in\{1,2\}$, the subgraph $Q_i$ is a triangle
with specified terminals $u_i,v_i$, or a copy
$T(x'_i,y'_i)$ with $\{x'_i,y'_i\}=\{u_i,v_i\}$.
The paths $P_1,P_2$ join $u_1,u_2$ and $v_1,v_2$,
respectively, and their interiors avoid $Q_1\cup Q_2$.
A path of length 0 identifies its ends.
These are the only identifications between $Q_1,Q_2$.

\item
The union $Q_1\cup Q_2\cup P_1\cup P_2$ is a spanning
subgraph of $G$.
Every edge of $G$ outside this subgraph joins $P_1$ and $P_2$.
The set $F$ of all edges of $G$ joining $P_1$ and $P_2$
satisfies $|F|\le2$.
\end{enumerate}
\end{lemma}
\begin{proof}
For a good cut $S_0$, the two components of $G-S_0$
are denoted by $B_0,D_0$.
Among good cuts disjoint from $V(B_0)$, a cut $S_1$
maximizes the order of the component containing $B_0$.
A cut $S_2$ is chosen analogously with $B_0,D_0$
interchanged.
The other components at these cuts are denoted by
$H_1,H_2$, respectively.
By Lemma~\ref{lem:terminal-side}, the induced subgraphs
\[
 Q_i=G[V(H_i)\cup S_i]\qquad(i=1,2)
\]
are triangles or copies of $T$ with terminal set $S_i$. Each vertex of $S_0$ has a neighbor in both $B_0,D_0$.
Consequently,
\[
 V(H_1)\subseteq V(D_0),\qquad
 V(H_2)\subseteq V(B_0),\qquad
 V(Q_1)\cap V(Q_2)=S_1\cap S_2.
\]
If $S_1=S_2$, terminal paths in $Q_1,Q_2$ of lengths
$(2,2)$, $(1,3)$ or $(4,4)$, according to their types,
form a cycle of length divisible by 4.
Thus $|S_1\cap S_2|\le1$.

By 2-connectivity, there are two vertex-disjoint paths
joining $S_1$ to $S_2$.
When these sets intersect, one path is trivial at their
common vertex.
Among all such pairs, the paths $P_1,P_2$ have minimum
total length.
The terminal labels are chosen so that
\[
 S_i=\{u_i,v_i\},\qquad
 P_1:u_1\longrightarrow u_2,\qquad
 P_2:v_1\longrightarrow v_2.
\]
Both paths are induced, and their interiors avoid
$Q_1\cup Q_2$.
Indeed, a path passing through $Q_i-S_i$ would use both
vertices of $S_i$, leaving none for the other path.

For $i=1,2$, the terminal residue set is
\[
 A_i=\mathcal R(Q_i;u_i,v_i)=
 \begin{cases}
 \{1,2\},&Q_i\cong K_3;\\
 \{0,3\},&Q_i\cong T.
 \end{cases}
\]
Every pair of vertex-disjoint paths joining $S_1$ to $S_2$,
with total length $t$, therefore satisfies
\(
 0\notin t+A_1+A_2.
\)
Hence
\begin{equation}\label{eq:structure-link-residue}
 t\equiv
 \begin{cases}
 3\pmod4,&\text{if $Q_1,Q_2$ have the same type};\\
 1\pmod4,&\text{otherwise}.
 \end{cases}
\end{equation}

If the union $Q_1\cup Q_2\cup P_1\cup P_2$ is not spanning,
a component $H$ outside this union has a neighbor set
$A\subseteq V(P_1\cup P_2)$ with $|A|\ge2$.
For $|A|\ge3$, a minimal tree joining three vertices of
$A$ through $H$ has three internally disjoint paths
from a vertex of $H$ to them.
Two of these paths have the same parity.
Their union is an even path with distinct ends in $A$
and interior in $H$.
For $|A|=2$, such an even path exists by the argument
in the proof of Lemma~\ref{lem:cuts}(ii).
In either case, this path is denoted by $R$.

If the ends of $R$ lie on different paths $P_1,P_2$,
the routes through $Q_1$ and $Q_2$ can each be chosen
even, since each $A_i$ contains both parities.
Together with $R$, these routes form an even theta,
a contradiction.
If both ends lie on the same path, its segment $L$
between them can be replaced by $R$.
By \eqref{eq:structure-link-residue},
\[
 \ell(L)\equiv\ell(R)\pmod4.
\]
Since $R$ is even, the cycle $L\cup R$ has length
divisible by 4, again a contradiction.
Thus $Q_1\cup Q_2\cup P_1\cup P_2$ is spanning.

Every vertex of $Q_i-S_i$ has all its neighbors in $Q_i$.
Since $P_1,P_2$ are induced, every edge outside the spanning
subgraph joins $P_1$ and $P_2$.
The set of all edges joining these paths is denoted by $F$.
Both paths are oriented from $Q_1$ towards $Q_2$, and
\(
 s=\ell(P_1)+\ell(P_2).
\)
If $ab,cd\in F$, with $a$ preceding $c$ on $P_1$
and $d$ preceding $b$ on $P_2$, then
\[
 L=\ell(P_1[a,c])+\ell(P_2[d,b])\ge2.
\]
For $L=2$, these segments and the two edges form a 4-cycle.
For $L>2$, the paths
\[
 P_1[u_1,a]\cup ab\cup P_2[b,v_2],
 \qquad
 P_2[v_1,d]\cup dc\cup P_1[c,u_2]
\]
are vertex-disjoint and have total length $s-L+2<s$.
Both cases are impossible.
Thus the edges of $F$ have the same weak order along
the two paths.

If $|F|\ge3$, three distinct edges $a_ib_i\in F$
$(i=1,2,3)$ are chosen in this order along both paths
so that
\[
 \ell(P_1[a_1,a_3])+\ell(P_2[b_1,b_3])
\]
is minimum.
Exactly these three edges join the two displayed segments.
Otherwise, three consecutive edges among four would
give a smaller sum.

For $i=1,2$, the subgraph
\[
 C_i=P_1[a_i,a_{i+1}]
       \cup P_2[b_i,b_{i+1}]
       \cup\{a_ib_i,a_{i+1}b_{i+1}\}
\]
is a cycle.
If both $C_1,C_2$ are triangles, the outer edges
$a_1b_1,a_3b_3$ and the intervening segments form a 4-cycle.
After reversing both paths if necessary, we may therefore
assume that $\ell(C_1)\ge5$.

All internal vertices of the two path segments in $C_1$
have degree 2 in $G$.
If $a_1=a_2$ or $b_1=b_2$, the other segment has length
at least 3 and contains adjacent vertices of degree 2.
Otherwise, the absence of such vertices gives
\[
 1\le\ell(P_1[a_1,a_2])\le2,\qquad
 1\le\ell(P_2[b_1,b_2])\le2.
\]
At least one of these lengths is 2, since $\ell(C_1)\ge5$.

The ordering of the edges in $F$ and the choice of the
three edges imply that $\{a_1,b_2\}$ is a 2-vertex cut.
The component containing
$Q_1-\{a_1,b_2\}$ is denoted by $U$.
The only neighbor of $a_1$ outside $U$ is its successor
$a^+$ on $P_1$.
The only neighbor of $b_2$ in $U$ is its predecessor
$b^-$ on $P_2$.
In particular, $a_1b_2\notin E(G)$.
A switch at $\{a_1,b_2\}$ on $U$ produces a graph $G'$
with
\[
 N_{G'}(a_1)=\{a^+,b^-\}.
\]
At least one of $a^+,b^-$ is an internal vertex of a
segment of length 2 and still has degree 2 in $G'$.
Thus $G'$ has adjacent vertices of degree 2,
contrary to the hypothesis.
Consequently, $|F|\le2$.
\end{proof}

\begin{corollary}\label{cor:structure-counts}
For the decomposition in Lemma~\ref{lem:structure},
the parameters are given below, where
$s=\ell(P_1)+\ell(P_2)$ and $Q_1,Q_2$ may be interchanged.
\[
\begin{array}{c|c|c|c}
\text{types of }Q_1,Q_2&v(G)&e(G)&s\pmod4\\ \hline
K_3,K_3&s+4&s+6&3\\
K_3,T&s+7&\le s+11&1\\
T,T&s+10&\le s+16&3
\end{array}
\]
If $Q_1,Q_2$ have different types and $s=1$, then
$e(G)=11$.
\end{corollary}

\begin{proof}
The congruences follow from the residue calculation
in the proof of Lemma~\ref{lem:structure}.
The paths $P_1,P_2$ contain $s+2$ vertices and $s$ edges.
Apart from the edges in $F$, a triangle contributes
1 additional vertex and 2 edges, while a copy of $T$
contributes 4 additional vertices and 7 edges.
Together with $|F|\le2$, these counts give the table.
For two triangles, their terminal edges give $|F|=2$.

In the mixed case $s=1$, the paths are a trivial path
at $z$ and an edge $pq$.
Thus $F\subseteq\{zp,zq\}$.
One of these edges is the triangle's terminal edge.
The other joins the terminals of the induced copy of $T$
and is therefore absent.
Hence
\(
 e(G)=3+7+1=11.
\)
\end{proof}

The following analysis treats graphs without a good cut
and is also used in Section~\ref{sec:two-connected}.

\begin{lemma}\label{lem:core}
If a 2-connected $\mathcal C_{0\bmod4}$-free graph $G$ has no adjacent vertices of degree 2 and no good cut, then suppression of its $2$-degree vertices produces a simple 3-connected planar graph $K$. Each edge of $K$ corresponds to a path of length 1 or 2 in $G$.
\end{lemma}

\begin{proof}
Since $G$ has no adjacent vertices of degree 2, it is not a cycle, and each thread has length 1 or 2. The 2-connectivity of $G$ ensures that each thread has distinct ends of degree at least 3. Thus the multigraph $K$ is connected and loopless, has no cut vertex, and satisfies $\delta(K)\ge3$.

If two edges of $K$ are parallel, their corresponding threads cannot both have length 1, since $G$ is simple, or both have length 2, since $G$ contains no 4-cycle. They are therefore $xy$ and $xuy$ for some vertex $u$.
Since $G\ne K_3$, the pair $\{x,y\}$ is a 2-vertex cut of $G$.
The edge $xy$ makes this cut good by
Lemma~\ref{lem:cuts}, a contradiction.
Hence $K$ is simple and $v(K)\ge4$.

If $\{x,y\}$ were a 2-vertex cut of $K$,
every component of $K-\{x,y\}$ would have at least 2 vertices, since $K$ is simple and $\delta(K)\ge3$. Their expansions lie in distinct components of $G-\{x,y\}$, each with at least 2 vertices. Lemma~\ref{lem:cuts} implies that these are its only 2 components and that $\{x,y\}$ is a good cut, a contradiction. Thus $K$ is 3-connected.
The graph $G$ is planar by   Lemma~\ref{lem:planar}, and suppression preserves planarity. Therefore $K$ is planar.
\end{proof}

In a simple 3-connected plane graph, face boundaries are induced cycles, and two distinct face boundaries intersect
in at most one vertex or one edge
\cite[Propositions~4.2.6 and~4.2.7]{Diestel}
\cite[Lemma~3]{FHRV}.

\begin{lemma}\label{lem:faces}
If a 2-connected plane $\mathcal C_{0\bmod4}$-free graph $G$
has no adjacent vertices of degree 2 and no good cut, then
\begin{enumerate}[label=\rm(\roman*)]
\item
$G$ has at most one triangular face.
\item
Any two distinct pentagonal face boundaries intersect
in one vertex or a path of length 2.
\item
No three pentagonal face boundaries have singleton
pairwise intersections.
\item
If three pentagonal face boundaries intersect pairwise
in paths of length 2, then $G$ contains a triangle-star.
\end{enumerate}
Consequently, if $G$ is locally reduced throughout its
switching orbit, the pairwise intersections of any three
distinct pentagonal face boundaries are not all of
the same type.
\end{lemma}

\begin{proof}
For the embedding inherited from $G$, the graph $K$
in Lemma~\ref{lem:core} is a simple 3-connected plane graph.
Each edge of $K$ expands to a thread of length 1 or 2.
The preceding face-intersection property therefore gives
\[
 C\cap D=\varnothing
 \quad\text{or}\quad
 C\cap D\text{ is a path of length }0,1,\text{ or }2
\]
for distinct face boundaries $C,D$ of $G$.
A common path of length 2 is a thread.

For odd cycles $C,D$ with
$\ell(C)\equiv\ell(D)\pmod4$ and distinct vertices
$a,b\in V(C)$, $c,d\in V(D)$, we have
\begin{equation}\label{eq:face-arc-sums}
 \{0,2\}\subseteq
 \mathcal R(C;a,b)+\mathcal R(D;c,d).
\end{equation}
Indeed, the two even arc sums have total length
$\ell(C)+\ell(D)\equiv2\pmod4$.
Thus two disjoint such cycles cannot be joined by
3 vertex-disjoint paths, since two of these paths
have an even total length.
The same identity excludes an even path between
$C-x$ and $D-x$ whose interior avoids $C\cup D$
when $V(C)\cap V(D)=\{x\}$.

(i) Two triangular faces sharing an edge form a 4-cycle. A triangular boundary contains no thread of length 2, since its suppression would produce parallel edges in $K$. Thus triangular boundaries are also triangles in $K$.

If two such boundaries are disjoint, 3-connectivity of $K$ supplies 3 vertex-disjoint paths between them. Their expansions contradict the preceding observation. 
If the boundaries are $xabx$ and $xcdx$, the graph $K-x$ is 2-connected. Menger's theorem supplies two vertex-disjoint paths
between $\{a,b\}$ and $\{c,d\}$.
After relabelling, their expansions are an $(a,c)$-path $P$ and a $(b,d)$-path $Q$, with interiors outside the triangles.
For $p=\ell(P)$, the two triangles and $P$ contain cycles of lengths $p+2,p+3,p+4$.
Hence $p\equiv3\pmod4$, and similarly
$\ell(Q)\equiv3\pmod4$.
But $P\cup Q\cup\{ab,cd\}$ is then a cycle of length $p+\ell(Q)+2\equiv0\pmod4$.

(ii) Two pentagonal boundaries sharing one edge form an 8-cycle. The argument for disjoint triangles applies equally
to disjoint pentagonal boundaries.
The intersection property above leaves only
a common vertex or a common thread of length 2.

(iii) For three pentagonal boundaries $C_1,C_2,C_3$ with singleton pairwise intersections, the intersection vertices are either all distinct or all equal.
In the first case, the two arcs between the intersection vertices on each $C_i$ have residue set $\{r_i,r_i+1\}$ for some $r_i$.
Independent choices of these arcs form cycles with residues
\[
 \sum_{i=1}^{3}\{r_i,r_i+1\}
 =r_1+r_2+r_3+\{0,1,2,3\}
 =\{0,1,2,3\},
\]
a contradiction.

In the second case, their common vertex $x$ belongs to $K$, since a $2$-degree vertex is incident with only 2 faces. Deletion of $x$ merges its incident faces into a face
of the 2-connected plane graph $K-x$.
Its boundary is a cycle $\Gamma$.
The corresponding boundaries of $C_1,C_2,C_3$ in $K$, with $x$ deleted, are vertex-disjoint nontrivial subpaths of $\Gamma$.

The three intervening arcs of $\Gamma$ expand to paths $Q_1,Q_2,Q_3$ whose interiors avoid all three boundaries. Their ends are labelled so that $Q_i$ joins $b_i\in V(C_i)$ to $a_{i+1}\in V(C_{i+1})$, with indices taken modulo 3. Each $Q_i$ is odd by the observation following
\eqref{eq:face-arc-sums}.
For the $(a_i,b_i)$-arc $A_i$ of $C_i$ avoiding $x$, an even length would make
$Q_{i-1}\cup A_i\cup Q_i$ an even path between
the other two boundaries, again a contradiction. Thus all $A_i,Q_i$ are odd.

Their union is a cycle in which each segment
$Q_i\cup A_{i+1}$ between $b_i,b_{i+1}$ is even. The three arcs from $x$ to $b_i$ in $C_i$ avoiding $a_i$ meet this cycle only at their respective ends. Two of these arcs have the same parity. Their union and the two intervening routes on the cycle
form an even theta, a contradiction.

(iv) The shared threads $C_1\cap C_2$ and $C_1\cap C_3$ have no common edge, since an edge is incident with only 2 faces.
Their vertex sets intersect because each contains 3 vertices of the pentagon $C_1$.
A common vertex cannot be internal to either thread, since it would have degree 2 and belong to 3 faces. Two common ends would give a 4-cycle. Thus the threads have the forms
\[
 C_1\cap C_2=zua,\qquad C_1\cap C_3=zvb,
\]
and the remaining edge of $C_1$ is $ab$.

If $C_2\cap C_3$ avoids $z$, its ends must be $a,b$. The remaining edge of $C_2$ is then $zb$, giving the 4-cycle $zuabz$.
Hence
\[
 C_2\cap C_3=zwc.
\]
The vertices $u,v,w$ are distinct and have degree 2. The vertices $a,b,c$ are distinct, since equality would create a 4-cycle.
None is an internal vertex of another shared thread, since that vertex would belong to 3 faces.
At $z$, the three faces make each pair of
$zu,zv,zw$ consecutive in the cyclic order.
Therefore $\deg_G(z)=3$. The remaining facial edges are $ab,bc,ca$, so these vertices and edges form a triangle-star.

The final assertion follows from (iii) and (iv), since a locally reduced graph contains no triangle-star.
\end{proof}
For Lemmas~\ref{lem:pentagons} and~\ref{lem:nogood},
the graph $G\ne K_3$ is
2-connected and $\mathcal C_{0\bmod4}$-free, is locally
reduced throughout its switching orbit, and has no good cut.
For a fixed plane embedding, the numbers of vertices,
edges and faces are abbreviated to $v,e,f$, and the
deficits to $D,d$.

The graph $K$ in Lemma~\ref{lem:core} inherits a plane embedding from $G$.
Its \emph{plane dual} $K^*$ has one vertex for each face of $K$, including the unbounded face. Each edge $e\in E(K)$ corresponds to an edge $e^*\in E(K^*)$ whose ends represent the two faces incident with $e$. Each dual vertex lies in its corresponding face, and $e^*$ crosses $e$ once and no other edge of $K$.
The graph $K^*$ is simple and 3-connected;
see~\cite{Diestel}. A loop or parallel pair in $K^*$ would correspond to a bond of $K$ of size at most 2, which 3-connectivity excludes.
For each $a\in E(K^*)$, its weight $w(a)\in\{1,2\}$ is the length of the corresponding thread in $G$. For $F\subseteq E(K^*)$ and $x\in V(K^*)$, the total weight
and weighted degree are
\[
 w(F)=\sum_{a\in F}w(a),\qquad
 d_w(x)=\sum_{a\ni x}w(a).
\]
The weighted degrees are the face lengths of $G$.

A \emph{bond} is an inclusion-minimal nonempty edge cut.
For $\varnothing\ne X\subsetneq V(K^*)$, its edge cut is
\[
 \delta_{K^*}(X)=\{xy\in E(K^*):x\in X,\ y\notin X\}.
\]
This is a bond exactly when $K^*[X]$ and $K^*-X$ are connected;
see~\cite[Section~1.9]{Diestel}.
By planar cycle--cut duality, a bond of $K^*$ corresponds to
a cycle of $K$; see~\cite[Proposition~4.6.1]{Diestel}.
Its expansion in $G$ has length equal to the bond weight,
which is therefore nonzero modulo 4.
The cut weights satisfy
\begin{equation}\label{eq:dual-cut-weight}
 w\bigl(\delta_{K^*}(X)\bigr)
 =\sum_{x\in X}d_w(x)-2w\bigl(E(K^*[X])\bigr).
\end{equation}
For every $xy\in E(K^*)$, the cut $\delta_{K^*}(\{x,y\})$ is
a bond by 3-connectivity, and its weight is
\begin{equation}\label{eq:dual-pair-weight}
 d_w(x)+d_w(y)-2w(xy).
\end{equation}

\begin{lemma}\label{lem:pentagons}
Let $G$ be a 2-connected plane $\mathcal C_{0\bmod4}$-free graph. If $G$ is locally reduced throughout its switching orbit and has no good cut, then $G$ has at most 4 pentagonal faces and satisfies $d(G)\ge2$.
\end{lemma}

\begin{proof}
If $G$ has at least 5 pentagonal faces, any 5 of them
define a red--blue edge coloring of $K_5$.
Red edges correspond to shared paths of length 2,
and blue edges to singleton intersections.
By Lemma~\ref{lem:faces}, neither color contains a triangle.
Any vertex with 3 neighbors of one color would yield
a triangle of the other color.
Hence every vertex has degree 2 in each color,
and each color forms a 5-cycle.

For the corresponding set $S\subseteq V(K^*)$, we have
\[
 |S|=5,\qquad K^*[S]\cong C_5,\qquad
 w(a)=2\quad(a\in E(K^*[S])),\qquad
 d_w(x)=5\quad(x\in S).
\]
Each vertex of $S$ is therefore incident with exactly
one edge of $\delta_{K^*}(S)$, whose weight is 1.
In particular,
\[
 K^*-S\ne\varnothing,\qquad |\delta_{K^*}(S)|=5.
\]
By 3-connectivity, every component of $K^*-S$ has at least
3 neighbors in $S$.
Thus $K^*-S$ is connected, since otherwise
$|\delta_{K^*}(S)|\ge6$.

For a set $X$ of 4 consecutive vertices of $K^*[S]$,
the graph $K^*[X]$ is a path.
The graph $K^*-X$ is also connected, since the remaining
vertex of $S$ has a neighbor in $K^*-S$.
Hence $\delta_{K^*}(X)$ is a bond.
However, \eqref{eq:dual-cut-weight} gives
\[
 w\bigl(\delta_{K^*}(X)\bigr)
 =\sum_{x\in X}d_w(x)-2w\bigl(E(K^*[X])\bigr)
 =4\cdot5-2\cdot3\cdot2=8,
\]
contrary to $\mathcal C_{0\bmod4}$-freeness.

The numbers $f_3,f_5$ of triangular and pentagonal faces
satisfy $f_3\le1$ and $f_5\le4$.
Since every other face has length at least 6,
\[
 2e\ge6f-3f_3-f_5\ge6f-7.
\]
The integrality of $e$ yields $e\ge3f-3$.
Euler's formula therefore gives
\[
 d(G)=3v-1-2e=e-3f+5\ge2.
\]
\end{proof}

\section{The general case}\label{sec:completeness}

\begin{lemma}\label{lem:nogood}
Let $G$ be a 2-connected
$\mathcal C_{0\bmod4}$-free graph of order $\ge 4$. If $G$ is locally reduced throughout its switching orbit,
has no good cut and satisfies $D(G)\le11$,
then $G\cong R_9$.
\end{lemma}

\begin{proof}
Since $K^*$ is simple and 3-connected, $f=v(K^*)\ge4$.
Euler's formula and the definitions of $D,d$ give
\[
 e=3f+d-5,\qquad v=2f+d-3,\qquad D=2f+7d-16.
\]
By Lemma~\ref{lem:pentagons}, $d\ge2$.
Together with $D\le11$ and $f\ge4$, these identities imply
\[
 d=2,\qquad 4\le f\le6,\qquad
 e=3f-3,\qquad v=2f-1.
\]

If $G$ had no triangular face, Lemma~\ref{lem:pentagons}
would give
\[
 2e\ge6f-4>6f-6=2e.
\]
Hence $G$ has a unique triangular face by
Lemma~\ref{lem:faces}.
The remaining face lengths are at least 5 and sum to
$6f-9$, so the possibilities are
\[
\begin{array}{c|l}
f&\text{face lengths}\\ \hline
4&(3,5,5,5)\\
5&(3,5,5,5,6)\\
6&(3,5,5,5,5,7)\ \text{or}\ (3,5,5,5,6,6).
\end{array}
\]
The vertex of $K^*$ corresponding to the triangular face
is denoted by $t$.
Since $d_w(t)=3$ and $\delta(K^*)\ge3$,
\[
 \deg_{K^*}(t)=3,\qquad
 w(tx)=1\quad(x\in N_{K^*}(t)).
\]
For every $xy\in E(K^*)$ with $d_w(x)=d_w(y)=5$,
we have $w(xy)=2$.
Otherwise, \eqref{eq:dual-pair-weight} would give
a bond of weight $5+5-2=8$.

\medskip\noindent
\emph{Case 1. $f=4$.}
The graph $K^*$ is isomorphic to $K_4$, and all three
edges of $K^*-t$ have weight 2.
The three pentagonal face boundaries therefore intersect
pairwise in paths of length 2.
By Lemma~\ref{lem:faces}, $G$ contains a triangle-star,
contrary to local reduction.

\medskip\noindent
\emph{Case 2. $f=5$.}
The vertices of $K^*-t$ are labelled $p,q,r,a$ so that
\[
 d_w(p)=d_w(q)=d_w(r)=5,\qquad d_w(a)=6.
\]
The parameters
\[
 R=|E(K^*[\{p,q,r\}])|,\qquad
 \varepsilon=\mathbf1_{\{ta\in E(K^*)\}}
\]
satisfy
\[
 15=4R+(3-\varepsilon)+(6-\varepsilon),
 \qquad 4R=6+2\varepsilon.
\]
Thus $\varepsilon=1$ and $R=2$.
The labels $p,q,r$ are chosen so that
\[
 E(K^*[\{p,q,r\}])=\{pq,qr\}.
\]
Since
\[
 N_{K^*}(p)\subseteq\{q,t,a\},\qquad
 N_{K^*}(r)\subseteq\{q,t,a\},\qquad
 \delta(K^*)\ge3,
\]
the weighted degrees of $p,r$ give
\[
 w(tp)=w(tr)=1,\qquad w(ap)=w(ar)=2.
\]
Consequently, $N_{K^*}(t)=\{a,p,r\}$, and $d_w(q)=5$
forces $w(aq)=1$.
Thus $K^*$ is a wheel with hub $a$ and rim $tpqrt$, with
\[
\begin{aligned}
 (w(tp),w(pq),w(qr),w(rt))&=(1,2,2,1),\\
 (w(at),w(ap),w(aq),w(ar))&=(1,2,1,2).
\end{aligned}
\]
In the dual wheel $K$, the rim vertices $A,B,C,D$
correspond respectively to the faces $apq,aqr,art,atp$,
and the hub $h$ corresponds to the rim face of $K^*$.
The four edges subdivided in $G$ are $BC,DA,hA,hB$.
Hence $G\cong R_9$.

\medskip\noindent
\emph{Case 3. The face lengths are $(3,5,5,5,5,7)$.}
The vertex of weighted degree 7 is denoted by $b$.
An edge $tb$ would have weight 1 and give a bond of
weight $3+7-2=8$, so $tb\notin E(K^*)$.
Every neighbor $x$ of $b$ therefore has $d_w(x)=5$.
By \eqref{eq:dual-pair-weight}, $w(bx)=1$, since
weight 2 would give $7+5-4=8$.
It follows that
\[
 7=d_w(b)=\deg_{K^*}(b)\le4,
\]
a contradiction.

\medskip\noindent
\emph{Case 4. The face lengths are $(3,5,5,5,6,6)$.}
The vertices of $K^*-t$ are labelled $p,q,r,a,b$ so that
\[
 d_w(p)=d_w(q)=d_w(r)=5,\qquad
 d_w(a)=d_w(b)=6.
\]
If $ab\in E(K^*)$, then $w(ab)=1$, since weight 2 would
give $6+6-4=8$ in \eqref{eq:dual-pair-weight}.
The parameters
\[
 R=|E(K^*[\{p,q,r\}])|,\qquad
 s=|N_{K^*}(t)\cap\{p,q,r\}|,\qquad
 z=\mathbf1_{\{ab\in E(K^*)\}}
\]
satisfy $R\in\{1,2\}$ by Lemma~\ref{lem:faces}.
Moreover, $\deg_{K^*}(t)=3$ gives $s\ge1$.
The total weight $A$ of the edges joining $\{p,q,r\}$
to $\{a,b\}$ satisfies
\[
 15=4R+s+A,\qquad
 12=3-s+2z+A.
\]
Hence
\[
 s=3-2R+z.
\]
Since $s\ge1$ and $z\le1$, we obtain $R=1$ and $s=1+z$.
The labels $p,q,r$ are chosen so that
\[
 E(K^*[\{p,q,r\}])=\{pq\}.
\]
The inclusions $N_{K^*}(r)\subseteq\{t,a,b\}$ and
$\deg_{K^*}(r)\ge3$, together with $d_w(r)=5$, give
\[
 N_{K^*}(r)=\{t,a,b\},\qquad
 w(tr)=1,\qquad w(ra)=w(rb)=2.
\]

If $z=0$, then $s=1$ and $N_{K^*}(t)=\{r,a,b\}$.
For each $x\in\{p,q\}$, both $xa,xb$ are present,
since the edge $pq$ contributes only weight 2
to $d_w(x)=5$.
At $a$, we have
\[
 w(ap)+w(aq)=6-w(at)-w(ar)=3.
\]
After interchanging $p,q$ if necessary, $w(aq)=2$.

If $z=1$, then $s=2$, and the labels may be chosen
so that $N_{K^*}(t)=\{r,p,a\}$.
Since $tq\notin E(K^*)$ and $w(pq)=2$,
both $qa,qb$ are present.

In either case, the following table specifies a set $X$
and a spanning tree of $K^*-X$.
\[
\begin{array}{c|c|c|c}
z&X&E(K^*[X])&\text{spanning tree of }K^*-X\\ \hline
0&\{t,a,q\}&\{ta,aq\}&\{bp,br\}\\
1&\{t,r,b\}&\{tr,rb\}&\{aq,pq\}
\end{array}
\]
Both $K^*[X]$ and $K^*-X$ are connected,
so $\delta_{K^*}(X)$ is a bond.
In each row,
\[
 \sum_{x\in X}d_w(x)=14,\qquad
 w(E(K^*[X]))=3.
\]
Equation~\eqref{eq:dual-cut-weight} now gives
\[
 w(\delta_{K^*}(X))=14-2\cdot3=8,
\]
a contradiction.

\medskip\noindent
Only Case 2 remains, and hence $G\cong R_9$.
\end{proof}

\begin{lemma}\label{lem:reduce}
If $G$ is a 2-connected $\mathcal C_{0\bmod4}$-free graph
with $D(G)\le11$, then $G$ is switching-equivalent to one of
$K_3,S_7,L_8,R_9,L_{13}$, or to a graph containing adjacent
vertices of degree 2.
Unless $G=K_3$, every such pair forms the internal vertices
of a 3-edge ear with distinct ends.
\end{lemma}
\begin{proof}
The assertion holds for $G=K_3$.
For adjacent $2$-degree vertices $u,v$ in a graph
$G'\ne K_3$ switching-equivalent to $G$, their other
neighbors $a,b$ satisfy
\[
 N_{G'}(u)=\{a,v\},\qquad N_{G'}(v)=\{u,b\}.
\]
If $a=b$, this vertex is a cut vertex of $G'$,
contrary to 2-connectivity.
Thus $auvb$ is a 3-edge ear with distinct ends.
It remains to consider $G\ne K_3$ whose switching orbit
contains no graph with adjacent vertices of degree 2.

If a graph $G'$ in this orbit contains a triangle-star,
Lemma~\ref{lem:triangle-star} gives a 2-connected
$\mathcal C_{0\bmod4}$-free graph $H$ with
\[
 d(H)=d(G),\qquad D(H)=D(G)-4\le7.
\]
For $H\ne K_3$, the same lemma gives $d(H)\ge4$, whence
\[
 D(H)=v(H)-13+6d(H)\ge v(H)+11>7,
\]
a contradiction.
Hence $H=K_3$, and restoring the deleted triangle-star
gives $G'\cong S_7$.
Only the case in which $G$ is locally reduced throughout
its switching orbit remains.

If $G$ has no good cut, Lemma~\ref{lem:nogood} gives
$G\cong R_9$.
Otherwise, the subgraphs $Q_1,Q_2$ and paths $P_1,P_2$
are as in Lemma~\ref{lem:structure}.
The paths join $u_1,u_2$ and $v_1,v_2$, respectively,
and are labelled so that $\ell(P_1)\le\ell(P_2)$.
For $s=\ell(P_1)+\ell(P_2)$,
Corollary~\ref{cor:structure-counts} gives
\[
\begin{array}{c|c|c|c}
\text{types of }Q_1,Q_2&D(G)&d(G)&s\pmod4\\ \hline
K_3,K_3&7s-15&s-1&3\\
K_3,T&\ge7s-18&\ge s-2&1\\
T,T&\ge7s-21&\ge s-3&3
\end{array}
\]
The bound $D(G)\le11$ therefore forces $s=3,1,3$,
respectively.

If $Q_1\cong Q_2\cong K_3$, then $u_1v_1$ and $u_2v_2$
are the only edges joining $P_1$ and $P_2$.
The path lengths are $(0,3)$ or $(1,2)$.
For $(0,3)$, the internal vertices of $P_2$ form
an adjacent $2$-degree pair.
For $(1,2)$, the subgraphs and paths have the form
\[
 Q_i=u_iv_ir_iu_i\quad(i=1,2),\qquad
 P_1=u_1u_2,\qquad P_2=v_1pv_2.
\]
The components of $G-\{u_2,v_1\}$ have vertex sets
$\{v_2,r_2,p\}$ and $\{u_1,r_1\}$.
A switch on the first component at $\{u_2,v_1\}$
produces a graph $G'$ with
\[
 E(G')=
 \bigl(E(G)\setminus\{u_2v_2,u_2r_2,v_1p\}\bigr)
 \cup\{v_1v_2,v_1r_2,u_2p\}.
\]
In particular,
\(
 N_{G'}(u_2)=\{u_1,p\}\) and \(
 N_{G'}(p)=\{u_2,v_2\}.
\)
Both possibilities contradict the assumption
on the switching orbit.

If $Q_1\cong K_3$ and $Q_2\cong T$, then $s=1$ and
$e(G)=11$ by Corollary~\ref{cor:structure-counts}.
The vertices can be labelled so that
\[
 Q_1=xrsx,\qquad P_1=x,\qquad P_2=ry,
\]
where $x,y$ are the terminals of $Q_2$.
These subgraphs account for all 11 edges of $G$.
For $Q_2=T(x',y')$, the correspondence is either
\(
 (x',y')=(x,y)\) or \(
 (x',y')=(y,x).
\)
Both graphs are switching-equivalent to $L_8$ by
Remark~\ref{rem:l8-switch}.

Finally, if $Q_1\cong Q_2\cong T$, then $s=3$ and
\[
 v(G)=13,\qquad D(G)=6d(G).
\]
Since $d(G)\ge0$ by \eqref{eq:bound-b} and
$d(G)=38-2e(G)$ is even, $D(G)\le11$ forces
\[
 d(G)=0,\qquad e(G)=19.
\]
The subgraph $Q_1\cup Q_2\cup P_1\cup P_2$
has $2\cdot7+3=17$ edges.
Thus exactly two additional edges join $P_1$ and $P_2$.
For $i\in\{1,2\}$, the copy $Q_i=T(x_i,y_i)$ satisfies
\[
 (u_i,v_i)\in\{(x_i,y_i),(y_i,x_i)\}.
\]

If $(\ell(P_1),\ell(P_2))=(0,3)$, then
\[
 P_1=x,\qquad P_2=v_1pqv_2,\qquad u_1=u_2=x.
\]
For each $i$, the edge $xv_i$ would form a 4-cycle
with a terminal path of length 3 in $Q_i$.
The two additional edges are therefore $xp,xq$,
giving construction~A.

If $(\ell(P_1),\ell(P_2))=(1,2)$, then
\(
 P_1=u_1u_2\) and \( P_2=v_1pv_2.
\)
The same obstruction excludes $u_1v_1,u_2v_2$.
The remaining candidate edges are
\[
 e_1=u_1p,\qquad e_2=u_1v_2,\qquad
 e_3=u_2v_1,\qquad e_4=u_2p.
\]
The pairs $\{e_1,e_3\}$ and $\{e_2,e_4\}$ give
the 4-cycles
\(
 u_1u_2v_1pu_1\) and \( u_1u_2pv_2u_1,
\)
respectively.
The pair $\{e_2,e_3\}$, together with a terminal path
of length 3 in each $Q_i$, gives an 8-cycle.
Hence the two additional edges form one of
\[
 \{e_1,e_2\},\qquad \{e_3,e_4\},\qquad \{e_1,e_4\}.
\]
The first two pairs give construction~B, up to
interchanging $Q_1,Q_2$, and the last gives
construction~C.
Consequently, $G\in\mathcal Z$, and
Proposition~\ref{prop:switch} implies that $G$ is
switching-equivalent to $L_{13}$.
\end{proof}

The class $\mathcal G$ in Theorem~\ref{thm:main} is given by
the following construction.
\begin{definition}\label{def:class-G}
The class $\mathcal G$ consists of the graphs obtained from
$K_1,K_2,K_3,S_7,L_8,R_9,L_{13}$ by finitely many applications
of the following operations. Every input must already have
been constructed. An operation with 2 inputs uses disjoint
copies before the stated identification.
\begin{enumerate}[label=\textbf{G\arabic*.},ref=G\arabic*]
\item\label{rule:G1}
The 2 inputs are joined by identifying one vertex of each,
with no other identifications.

\item\label{rule:G2}
A Whitney switch is performed in a 2-connected block.

\item\label{rule:G3}
For distinct vertices $a,b\in V(H)$ of an input $H$ with
$1\notin\mathcal R(H;a,b)$, a 3-edge ear is added between them.
\end{enumerate}
The output $H$ of every operation must satisfy $D(H)\le11$.
\end{definition}

\begin{lemma}\label{lem:bookkeeping}
Let $G$ be a nonempty $\mathcal C_{0\bmod4}$-free graph.
\begin{enumerate}[label=\rm(\roman*)]
\item
If $G$ has $c$ components, then $D(G)\ge19(c-1)$.
If $G$ is connected, then its block set $\mathfrak B(G)$
satisfies
\[
 \sum_{B\in\mathfrak B(G)}(v(B)-1)=v(G)-1,\qquad
 D(G)=\sum_{B\in\mathfrak B(G)}D(B).
\]
Consequently, every extremal graph is connected,
and all its blocks are extremal.

\item
If $G$ is 2-connected, $D(G)\le11$, and
$uv\in E(G)$ satisfies $\deg_G(u)=\deg_G(v)=2$,
then $H=G-\{u,v\}$ is connected and extremal, with
\(
 D(H)=D(G)-2.
\)
\end{enumerate}
\end{lemma}

\begin{proof}
(i) For the components $G_1,\ldots,G_c$,
\eqref{eq:bound} gives
\[
 D(G)=19(c-1)+\sum_{i=1}^cD(G_i)\ge19(c-1).
\]
For connected $G$, the block-cut tree and the partition
of $E(G)$ into block edge sets give the two identities.
Since $D(B)\ge0$ for every block, $D(G)\le11$ forces
$c=1$ and $D(B)\le11$ for every $B\in\mathfrak B(G)$.

(ii) The graph $H$ is nonempty and
$\mathcal C_{0\bmod4}$-free.
The deletion removes 2 vertices and 3 edges, so
\[
 0\le D(H)=D(G)-38+36=D(G)-2\le9.
\]
Together with (i), this shows that $H$ is connected
and extremal.
\end{proof}

\begin{proof}[Proof of Theorem~\ref{thm:main}]
The seeds are connected and $\mathcal C_{0\bmod4}$-free,
and their deficits are at most 11 by
Table~\ref{tab:seeds} and $D(K_1)=0$.
Every operation preserves connectedness.
The operation G1 also preserves $\mathcal C_{0\bmod4}$-freeness,
because every cycle lies in one input.
For G2 this follows from the preservation properties of
Whitney switches, and for G3 from
Lemma~\ref{lem:sound} and $1\notin\mathcal R(H;a,b)$.

For an output $G$ with one input $H$ or 2 inputs $H_1,H_2$,
the vertex and edge counts give
\[
 D(G)=
 \begin{cases}
 D(H_1)+D(H_2),&\text{under G1};\\
 D(H),&\text{under G2};\\
 D(H)+2,&\text{under G3}.
 \end{cases}
\]
The condition on each output therefore gives
$0\le D(G)\le11$ for every $G\in\mathcal G$, which is
equivalent to extremality.

For the converse, membership in $\mathcal G$ is proved
by induction on $v(G)$ for an extremal graph $G$.
The case $G=K_1$ is a seed.
By Lemma~\ref{lem:bookkeeping}, the graph $G$ is connected
and all its blocks are extremal.
If $|\mathfrak B(G)|\ge2$, induction supplies a construction
of each block. Their attachment according to the block-cut
tree reconstructs $G$.
Every intermediate graph $H$ satisfies
\[
 D(H)=\sum_{B\in\mathfrak B(H)}D(B)
 \le\sum_{B\in\mathfrak B(G)}D(B)=D(G)\le11,
\]
so each application of G1 is permitted.
If the unique block is a bridge, then $G=K_2$ is a seed.

For a 2-connected graph $G$,  Lemma~\ref{lem:reduce}
supplies a switching-equivalent graph $G'$ that is a seed or contains adjacent vertices of degree 2. If $G'$ is a seed, inverse switches construct $G$ by G2. Otherwise, Lemma~\ref{lem:bookkeeping} supplies a connected
extremal graph $H$ with
\[
 v(H)=v(G)-2,\qquad D(H)=D(G)-2.
\]
Induction gives $H\in\mathcal G$.
The deleted path is a 3-edge ear with distinct ends $a,b$.
An $(a,b)$-path in $H$ of length $1\bmod4$ would complete a forbidden cycle with this ear, so $1\notin\mathcal R(H;a,b)$.
The operation G3 therefore reconstructs $G'$, with $D(G')=D(G)\le11$.
The graph $G$ is recovered by reversing the switches. Since the deficit remains equal to $D(G)\le11$ throughout, each application of G2 is permitted.
\end{proof}

\begin{corollary}\label{cor:block-order}
Let $G$ be a nonempty extremal $\mathcal C_{0\bmod4}$-free graph.
Then every block of $G$ has order at most $23$, and every block of even order has order at most $18$.
\end{corollary}

\begin{proof}
Let $B$ be a block of $G$. The assertion is clear if $B=K_2$.
Otherwise, $B$ is 2-connected. The additivity and
nonnegativity of $D$ give
\[
0\le D(B)\le D(G)\le11,
\]
while the 2-connected bound gives $d(B)\ge0$. Hence
\[
v(B)=D(B)+13-6d(B)\le24.
\]
If $v(B)$ is odd, then $v(B)\le23$.
If $v(B)$ is even, then $d(B)=3v(B)-1-2e(B)$ is odd,
so $d(B)\ge1$ and consequently $v(B)\le18$.
\end{proof}

\section{The 2-connected case}\label{sec:two-connected}

\begin{lemma}\label{lem:primitive-b}
Let $G$ be a 2-connected $\mathcal C_{0\bmod4}$-free graph.
If $d(G)\le1$, then $G$ is switching-equivalent to
$L_8$, $L_{13}$, or a graph containing adjacent vertices
of degree 2.
In the latter case, every such pair forms the internal
vertices of a 3-edge ear with distinct ends.
\end{lemma}

\begin{proof}
Since switches preserve $d$ and $d(K_3)=2$, every graph
in the switching orbit of $G$ is distinct from $K_3$.
For adjacent $2$-degree vertices in any such graph,
coincident remaining neighbors would give a cut vertex.
Thus each pair forms the internal vertices of a 3-edge
ear with distinct ends.
It remains to consider an orbit containing no such pair.

A triangle-star in a graph in this orbit would give,
by Lemma~\ref{lem:triangle-star}, a graph $H$ with
$d(H)=d(G)$ and either $H=K_3$ or $d(G)\ge4$.
Since $d(K_3)=2$, both alternatives contradict $d(G)\le1$.
Hence $G$ is locally reduced throughout its switching orbit.

If $G$ had no good cut, Lemma~\ref{lem:pentagons}
would give $d(G)\ge2$.
Thus Lemma~\ref{lem:structure} supplies the subgraphs
$Q_1,Q_2$ and paths $P_1,P_2$.
For $s=\ell(P_1)+\ell(P_2)$,
Corollary~\ref{cor:structure-counts} and $d(G)\le1$ give
\[
\begin{array}{c|c}
\text{types of }Q_1,Q_2&\text{conditions on }s\\ \hline
K_3,K_3&s\le2,\quad s\equiv3\pmod4\\
K_3,T&s\le3,\quad s\equiv1\pmod4\\
T,T&s\le4,\quad s\equiv3\pmod4
\end{array}
\]
The first row is impossible.
The remaining rows give $s=1$ and $s=3$, respectively,
and hence
\[
 v(G)\in\{8,13\},\qquad
 D(G)=v(G)-13+6d(G)\le6.
\]
Lemma~\ref{lem:reduce} now applies.
The absence of adjacent $2$-degree vertices throughout
the switching orbit excludes its last alternative.
Moreover,
\[
 d(K_3)=d(S_7)=d(R_9)=2
\]
by Table~\ref{tab:seeds}.
Thus $G$ is switching-equivalent to $L_8$ or $L_{13}$.
\end{proof}

A \emph{rooted graph} $(H;s,t)$ has an ordered pair of
distinct vertices $s,t$; a rooted copy preserves this order.
A \emph{cap} $S_7(s,t)$ is a rooted copy of $S_7$ with
$s$ at an internal vertex of an arm and $t$ at the centre.
Its internal vertex set is $V(S_7)\setminus\{s,t\}$.
The underlying graph is $F_7$ of Chu, Park and Ryu~\cite{CPR},
but our roots differ from those in their figure.
For this root pair,
\begin{equation}\label{eq:seven-cap-data}
 v(S_7)=7,\qquad e(S_7)=9,\qquad d(S_7)=2,\qquad
 \mathcal R(S_7;s,t)=\{0,1\}.
\end{equation}
Indeed, the $(s,t)$-paths have lengths $1,4,5$: besides $st$,
each uses the edge from $s$ to the triangle, 1 or 2 triangle
edges, and another arm.

\begin{definition}\label{def:class-B}
The class $\mathcal B$ consists of the graphs obtained from
$L_8,L_{13}$ by finitely many applications of B1--B4 below.
Each input must already have been constructed; two-input
operations use disjoint copies before the stated identification.
\begin{enumerate}[label=\textbf{B\arabic*.},ref=B\arabic*]
\item\label{rule:B1}
A Whitney switch is performed.

\item\label{rule:B2}
An input $H$ is extended by a 3-edge ear with ends $a,b$, where
\[
 a,b\in V(H),\qquad a\ne b,\qquad 1\notin\mathcal R(H;a,b).
\]

\item\label{rule:B3}
Two inputs $B_1,B_2$ have root pairs $(v_1,x),(v_2,y)$ satisfying
\[
 d(B_1)=d(B_2)=1,~ v_1x\in E(B_1),~v_2y\in E(B_2),
\text{ and } 
 \mathcal R(B_1;v_1,x)=\mathcal R(B_2;v_2,y)=\{1,2\}.
\]
The vertices $v_1,v_2$ are identified, and a 3-edge ear
is added between $x,y$.

\item\label{rule:B4}
Two inputs $K,B_1$ satisfy $d(K)=d(B_1)=1$.
The graph $K$ contains a cap $S_7(v,y)$ with
\[
 I=V(S_7(v,y))\setminus\{v,y\},\qquad
 \{e\in E(K):e\cap I\ne\varnothing\}\subseteq E(S_7(v,y)).
\]
The graph $B_2=K-I$ is 2-connected, and the root pairs
$(v_1,x)$ in $B_1$ and $(v,y)$ in $B_2$ satisfy
\[
 v_1x\in E(B_1),\quad vy\in E(B_2),\qquad
 \mathcal R(B_1;v_1,x)=\mathcal R(B_2;v,y)=\{1,2\}.
\]
In the disjoint union of $B_1,B_2$, the vertices $v_1,v$
are identified, and a 3-edge ear is added between $x,y$.
\end{enumerate}
\end{definition}

\begin{lemma}\label{lem:b-soundness}
Rules B1--B4 preserve 2-connectivity and
$\mathcal C_{0\bmod4}$-freeness.
Rules B1 and B2 preserve $d$.
The outputs $G$ of B3 and B4 satisfy $d(G)=0$ and $d(G)=1$,
respectively.
Consequently, every graph in $\mathcal B$ is bound-tight.
\end{lemma}

\begin{proof}
The assertions for B1 follow from the preservation
properties of Whitney switches.
For B2, the addition of an ear preserves 2-connectivity,
while Lemma~\ref{lem:sound} gives
$\mathcal C_{0\bmod4}$-freeness and $d(G)=d(H)$.

For B3 and B4, the graphs $B_1,B_2$ are 2-connected
and $\mathcal C_{0\bmod4}$-free.
After the stated identification,
\[
 V(B_1)\cap V(B_2)=\{v\},\qquad
 \mathcal R(B_1;v,x)=\mathcal R(B_2;v,y)=\{1,2\}.
\]
The output $G$ consists of $B_1\cup B_2$ and a 3-edge
ear $P$ with ends $x,y$.
Every cycle $C$ contained in neither $B_1$ nor $B_2$
has the form $P\cup Q_1\cup Q_2$.
Here $Q_1$ is a $(v,x)$-path in $B_1$, and $Q_2$
is a $(v,y)$-path in $B_2$.
Thus
\[
 \ell(C)\bmod4\in3+\{1,2\}+\{1,2\}=\{1,2,3\},
\]
so $G$ is $\mathcal C_{0\bmod4}$-free.

The graph $G-v$ is connected through $P$.
For $z\in V(G)\setminus\{v\}$, the surviving parts of
$B_1,B_2$ remain connected through $v$, and every
component of $P-z$ meets their union.
Hence $G$ is 2-connected.
Moreover,
\(
 v(G)=v(B_1)+v(B_2)+1\) \text{ and } \(
 e(G)=e(B_1)+e(B_2)+3,
\)
which gives
\begin{equation}\label{eq:rooted-join-deficit}
 d(G)=d(B_1)+d(B_2)-2.
\end{equation}
For B3, this yields $d(G)=0$.

In B4, the graph $K$ is the union of $B_2$ and the specified copy of $S_7(v,y)$.
Their intersection consists of the edge $vy$ and its ends.
Consequently,
\[
 v(B_2)=v(K)-5,\qquad e(B_2)=e(K)-8,\qquad
 d(B_2)=d(K)+1=2.
\]
Since $d(B_1)=1$, equation~\eqref{eq:rooted-join-deficit}
yields $d(G)=1$.

For a 2-connected $\mathcal C_{0\bmod4}$-free graph $H$,
the definition of $d$ and the integrality of $e(H)$ give
\[
 d(H)\in\{0,1\}
 \quad\Longleftrightarrow\quad
 e(H)=\left\lfloor\frac{3v(H)-1}{2}\right\rfloor.
\]
Since $d(L_8)=1$ and $d(L_{13})=0$, the last assertion
follows by induction on the number of operations.
\end{proof}

Figure~\ref{fig:construction-examples} gives two graphs of order 10
obtained from $L_8$. On the left, G1 attaches a triangle at $x$.
On the right, B2 adds the ear $xuvy$. Every $(x,y)$-path in
$L_8$ lies in $T$ or in the subgraph on $\{x,y,r,s\}$.
The paths in $T$ have lengths 3 and 4, and those in the latter
subgraph have lengths 2 and 3. Thus
$\mathcal R(L_8;x,y)=\{0,2,3\}$, as required by B2.
Both graphs have 14 edges and $D=3$, so they attain the general bound.
The graph on the left has cut vertex $x$; the graph on the right
is 2-connected and has $d=1$.

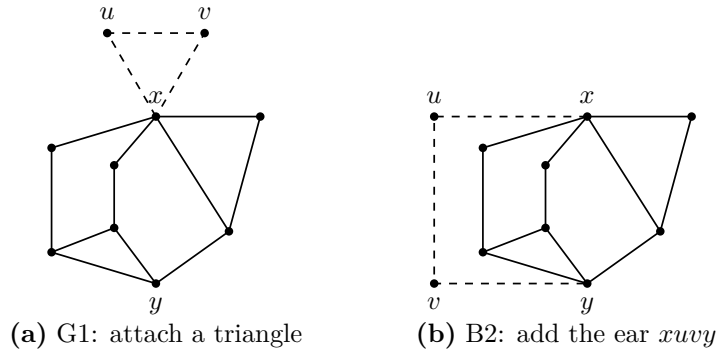
\begin{figure}[htbp]
\centering
\begin{tikzpicture}[
  x=.92cm,y=.92cm,line width=.65pt,
  every node/.style={font=\small},
  vertex/.style={circle,fill=black,inner sep=0pt,minimum size=3.2pt},
  every label/.append style={font=\small,inner sep=1pt},
  label distance=2pt
]
\foreach \j/\dx in {A/0,B/6.2}{
  \begin{scope}[shift={(\dx,0)}]
    \coordinate (\j-x) at (0,2.4);
    \coordinate (\j-a) at (-1.5,1.95);
    \coordinate (\j-b) at (-1.5,.45);
    \coordinate (\j-y) at (0,0);
    \coordinate (\j-c) at (-.6,.8);
    \coordinate (\j-d) at (-.6,1.7);
    \coordinate (\j-r) at (1.05,.75);
    \coordinate (\j-s) at (1.5,2.4);
    \draw (\j-x)--(\j-a)--(\j-b)--(\j-y)
          --(\j-c)--(\j-d)--cycle (\j-b)--(\j-c)
          (\j-x)--(\j-r)--(\j-s)--cycle (\j-r)--(\j-y);
    \foreach \z in {a,b,c,d,r,s}
      \node[vertex] at (\j-\z) {};
    \node[vertex,label=above:$x$] at (\j-x) {};
    \node[vertex,label=below:$y$] at (\j-y) {};
  \end{scope}
}
\coordinate (A-u) at (-.7,3.6);
\coordinate (A-v) at (.7,3.6);
\draw[dashed] (A-x)--(A-u)--(A-v)--cycle;
\node[vertex,label=above:$u$] at (A-u) {};
\node[vertex,label=above:$v$] at (A-v) {};
\node at (0,-.75) {\textbf{(a)} G1: attach a triangle};
\coordinate (B-u) at (4,2.4);
\coordinate (B-v) at (4,0);
\draw[dashed] (B-x)--(B-u)--(B-v)--(B-y);
\node[vertex,label=above:$u$] at (B-u) {};
\node[vertex,label=below:$v$] at (B-v) {};
\node at (5.9,-.75) {\textbf{(b)} B2: add the ear $xuvy$};
\end{tikzpicture}
\caption{Two constructions from $L_8$. }
\label{fig:construction-examples}
\end{figure}
\FloatBarrier

\begin{lemma}\label{lem:cap}
Let $B$ be a 2-connected $\mathcal C_{0\bmod4}$-free graph,
and let $a,b\in V(B)$ be distinct.
\begin{enumerate}[label=\rm(\roman*)]
\item
If $\mathcal R(B;a,b)$ is a mixed-parity pair, then
\[
 d(B)\ge
 \begin{cases}
  1,&\mathcal R(B;a,b)=\{1,2\};\\
  2,&\mathcal R(B;a,b)\in\{\{0,1\},\{2,3\}\};\\
  3,&\mathcal R(B;a,b)=\{0,3\}.
 \end{cases}
\]
Moreover, if $\mathcal R(B;a,b)=\{1,2\}$ and
$ab\notin E(B)$, then $d(B)\ge3$.
Consequently, if $d(B)=1$ and $\mathcal R(B;a,b)$
is a mixed-parity pair, then
\[
 \mathcal R(B;a,b)=\{1,2\},
 \qquad ab\in E(B).
\]
\item
If $d(B)=0$, then $0\in\mathcal R(B;a,b)$.
\end{enumerate}
\end{lemma}

\begin{proof}
Throughout the proof, $L$ denotes $\mathcal R(B;a,b)$.

(i) For $L=\{1,2\}$, the graph $H$ is obtained from $B$
and a disjoint cap $S_7(s,t)$ by identifying $s=a$
and $t=b$.
Any duplicated edge $ab$ is retained only once.
The graph $H$ is 2-connected.
Every cycle contained in neither $B$ nor the cap
consists of one $(a,b)$-path from each.
Its residue belongs to
\[
 L+\mathcal R(S_7;s,t)
 =\{1,2\}+\{0,1\}
 =\{1,2,3\}.
\]
Hence $H$ is $\mathcal C_{0\bmod4}$-free.
For $\epsilon=\mathbf1_{\{ab\in E(B)\}}$, the identities
\[
 v(H)=v(B)+5,\qquad e(H)=e(B)+9-\epsilon
\]
give
\[
 0\le d(H)=d(B)-3+2\epsilon.
\]
This proves both bounds for $L=\{1,2\}$.

The remaining bounds follow from the ear additions
in the following table, where $B'=B\cup P$ and all
internal vertices of $P$ are new.
Each addition preserves 2-connectivity.
Since $-\ell(P)\notin L$ in every row,
Lemma~\ref{lem:sound} gives
$\mathcal C_{0\bmod4}$-freeness and the stated change in $d$.
\[
\begin{array}{c|c|c|c|c}
 L&P&\text{root pair}&\text{residue set}&d(B')-d(B)\\ \hline
 \{0,1\}&azb&(a,z)&\{1,2\}&-1\\
 \{2,3\}&auvb&(u,v)&\{0,1\}&0\\
 \{0,3\}&azb&(a,z)&\{0,1\}&-1
\end{array}
\]
A path between the new roots is either their common edge
or an old $(a,b)$-path together with the corresponding
ear edges.
Thus the residue sets in the table are
$\{1\}\cup(L+1)$ for $P=azb$ and
$\{1\}\cup(L+2)$ for $P=auvb$.

The first row gives $d(B)-1\ge1$ for $L=\{0,1\}$.
This bound and the second row give $d(B)\ge2$
for $L=\{2,3\}$.
The third row then gives $d(B)-1\ge2$
for $L=\{0,3\}$.

(ii) Suppose that $d(B)=0$ and $0\notin L$.
The deficit identity gives
\[
 v(B)=D(B)+13\ge13.
\]
By Lemma~\ref{lem:bipartite-bound}, the graph $B$
is nonbipartite.
Lemma~\ref{lem:parity} implies that $L$ contains
both parities, while part (i) excludes $|L|=2$.
Consequently,
\[
 L=\{1,2,3\},\qquad ab\notin E(B).
\]
Indeed, the edge $ab$ and an $(a,b)$-path of residue 3
would form a forbidden cycle.

If $B-\{a,b\}$ is connected, the graph $H$ is obtained
by identifying $a,b$ as a single vertex $c$
and merging parallel edges.
Since $B$ contains no 4-cycle,
\[
 q=|N_B(a)\cap N_B(b)|\le1.
\]
Therefore
\[
 v(H)=v(B)-1\ge12,\qquad e(H)=e(B)-q,
\]
and
\[
 d(H)=d(B)-3+2q=-3+2q<0.
\]
The graph $H-c=B-\{a,b\}$ is connected.
For $z\in V(H)\setminus\{c\}$, the graph $H-z$
is obtained by identifying $a,b$ in the connected
graph $B-z$ and merging parallel edges.
Hence $H$ is 2-connected.
Every cycle of $H$ lifts to a cycle of $B$
or an $(a,b)$-path of the same length.
Since $0\notin L$, the graph $H$ is
$\mathcal C_{0\bmod4}$-free.
This contradicts \eqref{eq:bound-b}.

If $\{a,b\}$ is a 2-vertex cut, every side contains
an $(a,b)$-path, and every such path lies in one side.
The residues 1 and 3 must occur in the same side $A$,
since paths with these residues in different sides
would form a forbidden cycle.
Every other side has residue set $\{2\}$.
There is exactly one such side $C$, since two would
give a forbidden cycle of residue $2+2$.
The same obstruction excludes residue 2 from $A$.
Thus
\[
 \mathcal R(A;a,b)=\{1,3\},\qquad
 \mathcal R(C;a,b)=\{2\}.
\]

For a new vertex $z$, the graphs $A+ab$ and $C\cup azb$
are 2-connected.
Their $(a,b)$-paths are all odd and all even,
respectively.
Lemma~\ref{lem:parity} therefore implies that
$A$ and $C$ are bipartite.

Since $ab\notin E(B)$, an odd $(a,b)$-path in $A$
has length at least 3.
Hence $v(A)\ge4$, and Lemma~\ref{lem:bipartite-bound}
gives $d(A)\ge5$.
The side $C$ is either a 2-edge path, with $d(C)=4$,
or has order at least 4, with $d(C)\ge5$.
Since the two sides share only $a,b$ and have
disjoint edge sets,
\[
 d(B)=d(A)+d(C)-5\ge5+4-5=4,
\]
contrary to $d(B)=0$.
\end{proof}

\begin{proposition}\label{prop:ear-reduction}
Let $G$ be a 2-connected $\mathcal C_{0\bmod4}$-free graph. If $G$ is bound-tight and contains adjacent vertices of degree 2, then $G$ is obtained by B2, B3 or B4
from bound-tight graphs of order less than $v(G)$.
\end{proposition}

\begin{proof}
The adjacent $2$-degree vertices are denoted by $p,q$, with
\[
 N_G(p)=\{a,q\},\qquad N_G(q)=\{p,b\}.
\]
Since $d(G)\in\{0,1\}$ and $d(K_3)=2$, the graph $G$
is distinct from $K_3$.
If $a=b$, then $a$ is a cut vertex of $G$.
Hence $a\ne b$.

The graph $H=G-\{p,q\}$ is connected, since $q$
is a leaf of the connected graph $G-p$.
The 2-connectivity of $G$ implies that the block-cut tree
of $H$ is a path, with $a,b$ in its end blocks.
The blocks are indexed in order as $B_1,\ldots,B_m$, with
\[
 x_1=a,\qquad y_m=b,\qquad
 V(B_i)\cap V(B_{i+1})=\{y_i\}=\{x_{i+1}\}
 \quad(1\le i<m).
\]
The deletion removes 2 vertices and 3 edges,
so $d(H)=d(G)$.
For $L_i=\mathcal R(B_i;x_i,y_i)$, the block identities
and concatenation of block paths give
\begin{align}
 d(G)&=2+\sum_{i=1}^m\bigl(d(B_i)-2\bigr),
 \label{eq:ear-deficit}\\
 \mathcal R(H;a,b)&=\sum_{i=1}^m L_i,
 \qquad 1\notin\mathcal R(H;a,b).
 \label{eq:ear-residues}
\end{align}
The residue 1 is excluded because the corresponding path
and $apqb$ would form a forbidden cycle.

For bipartite blocks, the identity $d(K_2)=3$ and
Lemma~\ref{lem:bipartite-bound} give
\[
\begin{aligned}
 d(B_i)-2&=1 &&(B_i=K_2),\\
 d(B_i)-2&\ge3 &&(B_i\ne K_2).
\end{aligned}
\]
Every nonbipartite block supplies a mixed-parity pair
in $L_i$ by Lemma~\ref{lem:parity}.
For any three mixed-parity pairs $A_1,A_2,A_3$,
\[
 |A_1+A_2|\ge3,\qquad
 A_1+A_2+A_3=\mathbb Z/4\mathbb Z.
\]
Thus \eqref{eq:ear-residues} allows at most
two nonbipartite blocks.
There is at least one, since otherwise
\eqref{eq:ear-deficit} gives $d(G)\ge3$.
The total contribution of the bipartite blocks is
\[
 W=\sum_{i:\,B_i\text{ is bipartite}}
       \bigl(d(B_i)-2\bigr).
\]

\medskip\noindent
\emph{Case 1. $H$ has exactly one nonbipartite block $B$.}
Equation~\eqref{eq:ear-deficit} gives
\(
 d(B)+W=d(G)\le1.
\)
For $W=0$, there are no other blocks.
Hence $H=B$ is bound-tight, and B2 reconstructs $G$
from the smaller graph $H$.
For $W>0$, the only possibility is
\[
 (d(B),W,d(G))=(0,1,1),
\]
with exactly one bridge.
Lemma~\ref{lem:cap}(ii) gives residue 0 between
the roots of $B$.
Together with the bridge residue 1, this contradicts
\eqref{eq:ear-residues}.

\medskip\noindent
\emph{Case 2. $H$ has exactly two nonbipartite blocks.}
These blocks are denoted by $B_1,B_2$, with deficits
$d_1,d_2$ and root residue sets $L_1,L_2$.
Both sets are mixed-parity pairs.
Indeed, a set of at least 3 residues added to a
mixed-parity pair gives $\mathbb Z/4\mathbb Z$.
Lemma~\ref{lem:cap}(i) and \eqref{eq:ear-deficit} yield
\[
 d_1,d_2\ge1,\qquad
 d_1+d_2+W=d(G)+2\le3.
\]
Consequently, no bipartite block containing a cycle occurs.
After interchanging the two blocks if necessary,
\[
 (d_1,d_2,W)\in
 \{(1,1,0),(1,1,1),(1,2,0)\}.
\]

For $(d_1,d_2,W)=(1,1,0)$,
Lemma~\ref{lem:cap}(i) gives
\(
 L_1=L_2=\{1,2\},
\)
with both root pairs adjacent.
Rule B3 reconstructs $G$ from the strictly smaller
bound-tight graphs $B_1,B_2$.

The possibility $(d_1,d_2,W)=(1,1,1)$ is excluded by
\[
 \mathcal R(H;a,b)
 =\{1,2\}+\{1,2\}+\{1\}
 =\{0,1,3\},
\]
contrary to \eqref{eq:ear-residues}.

It remains to consider $(d_1,d_2,W)=(1,2,0)$.
Lemma~\ref{lem:cap}(i) gives $L_1=\{1,2\}$
with adjacent roots.
If $L_2\in\{\{0,1\},\{2,3\}\}$, then
$1\in L_1+L_2$, a contradiction.
The possibility $L_2=\{0,3\}$ would require $d_2\ge3$.
Hence $L_2=\{1,2\}$.
Its roots are also adjacent, since otherwise
Lemma~\ref{lem:cap}(i) gives $d_2\ge3$.

The blocks meet at $v$, and their other roots
are denoted by $x,y$, respectively.
The graph $K$ is the union of $B_2$ and a cap
$S_7(v,y)$, with intersection consisting of
the edge $vy$ and its ends.
The cap construction in Lemma~\ref{lem:cap}(i)
shows that $K$ is 2-connected and
$\mathcal C_{0\bmod4}$-free, with
\[
 d(K)=d(B_2)-1=1,\qquad
 v(K)=v(B_2)+5.
\]
Since $d(B_1)=1$,
\[
 D(B_1)=v(B_1)-7\ge0,\qquad
 v(B_1)\equiv0\pmod2.
\]
Thus $v(B_1)\ge8$, and
\[
 v(K)=v(B_2)+5
 <v(B_1)+v(B_2)+1=v(G).
\]
The graph $B_1$ also has order less than $v(G)$.

For $I=V(S_7(v,y))\setminus\{v,y\}$, the construction gives
\[
 K-I=B_2,\qquad
 \{e\in E(K):e\cap I\ne\varnothing\}
 \subseteq E(S_7(v,y)).
\]
Hence $K,B_1$ satisfy the conditions of B4
and are strictly smaller bound-tight inputs.
The resulting graph is $G$.
\end{proof}

\begin{proof}[Proof of Theorem~\ref{thm:b}]
Every member of $\mathcal B$ is bound-tight by
Lemma~\ref{lem:b-soundness}.
For the converse, the stronger assertion that every
bound-tight graph belongs to $\mathcal B$ is proved
by induction on its order.

For a bound-tight graph $G$, we have $d(G)\in\{0,1\}$.
By Lemma~\ref{lem:primitive-b}, some graph $G'$ in its
switching orbit is isomorphic to $L_8$ or $L_{13}$,
or contains adjacent vertices of degree 2.
In the first two cases, $G'\in\mathcal B$ by the
definition of the seeds.
Otherwise, $G'$ is also bound-tight, and
Proposition~\ref{prop:ear-reduction} reconstructs it
by B2, B3 or B4 from bound-tight graphs of smaller order.
These inputs belong to $\mathcal B$ by induction,
so $G'\in\mathcal B$.
In either case, the inverse switches permitted by B1
give $G\in\mathcal B$.
\end{proof}

\section*{Declaration on the Use of AI}

The proof was found by GPT-5.6 following some strategies supplied by the authors. This paper presents the proof and explains the main strategies and the intuition behind its key new ideas. The authors later replaced or simplified many of the arguments. Other arguments came from the literature.  All results and proofs were written by the authors, who take full responsibility for the content of this paper.

\section*{Acknowledgments}

Yanan Hu’s work was supported by the Science and Technology Commission of Shanghai Municipality Grant Number 25ZR1402474. Leilei Zhang’s work was supported by the National Natural Science Foundation of China (Grant Nos. 12601664, 12571365), JSPS KAKENHI (Grant No.~25KF0036), the Hubei Provincial Natural Science Foundation of China (Grant No.~2025AFB309), the China Postdoctoral Science Foundation (Grant No.~2025M773113), the Fundamental Research Funds for the Central Universities, Central China Normal University (Grant No.~CCNU24XJ026).  

\section*{ Declaration}
	
	\noindent\textbf{Conflict~of~interest}
	The authors declare that they have no known competing financial interests or personal relationships that could have appeared to influence the work reported in this paper.
	
	\noindent\textbf{Data~availability}
	No data were used for the research described in this article.

\end{document}